\documentclass{article}
\usepackage[left=1in,top=1in,right=1in,bottom=1in]{geometry}

\usepackage{graphicx,amsmath,amssymb,amsthm,color,tikz-cd}

\usepackage{bbold}

\usepackage{caption}
\usepackage{subcaption}
\usepackage{float}

\usepackage{pgfplots}
\pgfplotsset{compat=1.11}

\usepackage{enumitem}

\numberwithin{equation}{section}
\newtheorem{theorem}{Theorem}[section]
\newtheorem{lemma}[theorem]{Lemma}
\newtheorem{proposition}[theorem]{Proposition}
\newtheorem{definition}[theorem]{Definition}
\newtheorem{remark}[theorem]{Remark}
\newtheorem{question}{Question}
\newtheorem{corollary}[theorem]{Corollary}

\newtheorem{thmintro}{Theorem}

\newcommand{\eps}{\varepsilon}

\newcommand{\Z}{\mathbb Z}

\DeclareMathOperator{\sgn}{sgn}
\DeclareMathOperator{\Int}{Int}
\DeclareMathOperator{\dist}{dist}
\DeclareMathOperator{\Leb}{Leb}

\newcommand{\Addresses}{{
  \bigskip
  \footnotesize

  \textsc{Mathematics Department, Stony Brook University}
	\par\nopagebreak
	\textsc{Simons Center for Geometry and Physics, Stony Brook, NY}
	\par\nopagebreak
  \textit{E-mail address:} \texttt{alena.erchenko@stonybrook.edu}
}}

\title{Flexibility of Lyapunov exponents with respect to two classes of measures on the torus}
\author{Alena Erchenko}
\date{}

\begin{document}

\maketitle

\newsavebox{\smlmat}
\savebox{\smlmat}{$\left(\begin{smallmatrix}2&1\\1&1\end{smallmatrix}\right)$}

\begin{center}
 \emph{Dedicated to Anatole Katok}
\end{center}
\vspace{0.5cm}

\begin{abstract}
    We consider a smooth area-preserving Anosov diffeomorphism $f\colon \mathbb T^2\rightarrow \mathbb T^2$ homotopic to an Anosov automorphism $L$ of $\mathbb T^2$. It is known that the positive Lyapunov exponent of $f$ with respect to the normalized Lebesgue measure is less than or equal to the topological entropy of $L$, which, in addition, is less than or equal to the Lyapunov exponent of $f$ with respect to the probability measure of maximal entropy. Moreover, the equalities only occur simultaneously. We show that these are the only restrictions on these two dynamical invariants.
\end{abstract}

\tableofcontents

\section{Introduction}

The aim of the flexibility program is to study natural classes of smooth dynamical systems and to find \textit{constructive tools} to freely manipulate dynamical data inside a fixed class. The result described in this paper is another example demonstrating the flexibility principle in dynamical systems. 

\subsection{Anosov volume-preserving diffeomorphisms on tori}\label{section: anosov}

Consider an $n$-dimensional torus $\mathbb T^n = \mathbb R^n/\mathbb Z^n$, $n\geq 2$. Anosov volume-preserving $C^\infty$ (smooth) diffeomorphisms represent a natural class for studying flexibility questions. For any $f$ in this class the tangent bundle of $\mathbb T^n$ splits as a direct sum $T\mathbb T^n=E^u\oplus E^s$ of two $Df$-invariant subbundles $E^u$ (unstable) and $E^s$ (stable) such that $E^u$ is uniformly expanded by $Df$ and $E^s$ is uniformly contracted by $Df$. Moreover, by \cite[Theorem 18.6.1]{KatokHasselblatt}, $f$ is homotopic and topologically conjugate to an Anosov automorphism $L$ given by a hyperbolic matrix in $SL(n,\mathbb Z)$, i.e., a matrix in $SL(n, \mathbb Z)$ with no eigenvalues on the unit circle. 

The \textit{Lyapunov exponent} of $f$ at $x\in\mathbb T^n$, $\pmb v\in T_x\mathbb T^n\setminus\{\pmb 0\}$ is given by
\begin{equation}\label{def: Lyapunov exponent}
\lambda(f,x,\pmb v) = \limsup\limits_{n\rightarrow\infty}\frac{\log\|Df^n_x\pmb v\|}{n}.
\end{equation}

Since $f$ is volume-preserving, $f$ preserves the normalized Lebesgue measure $\Leb$ which is ergodic \cite{AnosovSinai}. Moreover, there exists a unique measure of maximal entropy $\nu_f$ for $f$ which is ergodic \cite{Bowen}. Applying the Oseledets Multiplicative Ergodic Theorem, we obtain that for any $f$-invariant ergodic Borel probability measure $\mu$ the limit 
\begin{equation*}
\lim\limits_{n\rightarrow \infty}\frac{\log\|D_xf^n\pmb v\|}{n}
\end{equation*}
exists for $\mu$-almost every $x\in\mathbb T^n$ and all non-zero $\pmb v\in T_x\mathbb T^n$. In particular, we obtain a collection of $n$ numbers $\lambda_{1,\mu}(f)\geq\ldots\geq \lambda_{n,\mu}(f)$ as possible limits that are called \textit{Lyapunov exponents of $f$ with respect to $\mu$}. Since $f$ is Anosov, all these numbers are nonzero. We denote the number of positive elements in this list by $u(f)$ and call it the \textit{unstable index of $f$}. Also, for any $f$-invariant ergodic Borel probability measure $\mu$:
\begin{equation}\label{Lyapunov exponent through integral}
\sum\limits_{i=1}^{u(f)}\lambda_{i,\mu}(f) = \int_{\mathbb T^n}\log\|Df|_{E^u}\|d\mu.
\end{equation}

The Lyapunov exponents with respect to $Leb$ and $\nu_f$ are natural dynamical data to look at in the context of the flexibility program. 

We denote by $h_{top}(f)$ and $h_{\Leb}(f)$ the topological entropy and the metric entropy with respect to $\Leb$ of $f$, respectively. The following relations for the considered entropies and Lyapunov exponents are known in this setting:
\begin{itemize}
\item For an Anosov automorphism $L\colon\mathbb T^n\rightarrow\mathbb T^n$, we have $h_{top}(L)=h_{\Leb}(L)=\sum\limits_{i=1}^{u(L)}\lambda_{i,\Leb}(L)$;
\item Since $h_{top}$ is an invariant of topological conjugacy \cite[Corollary 3.1.4]{KatokHasselblatt}, we obtain that if $f$ is homotopic to an Anosov automorphism $L$, then $h_{top}(f)=h_{top}(L)$;
\item Since $f$ is volume-preserving, we have $\sum\limits_{i=1}^n\lambda_{i,\Leb}(f)=\sum\limits_{i=1}^n\lambda_{i,\nu_f}(f)=0$;
\item Variational Principle for entropies \cite[Theorem 4.5.3]{KatokHasselblatt}: $h_{\Leb}(f)\leq h_{top}(f)$; 
\item Ruelle's inequality \cite{Ruelle}: $h_{top}(f)\leq \sum\limits_{i=1}^{u(f)}\lambda_{i,\nu_f}(f)$;
\item Pesin's entropy formula \cite[Theorem 10.4.1]{BarreiraPesin}: $h_{\Leb}(f) = \sum\limits_{i=1}^{u(f)}\lambda_{i,\Leb}(f)$.
\end{itemize}

Thus, some representative questions (ordered in increasing number of requirements) concerning flexibility of Lyapunov exponents for Anosov volume-preserving diffeomorphisms on $\mathbb T^n$ are the following:

\begin{question}[Conjecture 1.4 in \cite{BKRH}, weak flexibility for one measure]\label{weak flexibility 1}
Given any list of nonzero numbers $\xi_1\geq\ldots\geq \xi_n$ such that $\sum\limits_{i=1}^n\xi_i=0$, does there exist a smooth volume-preserving Anosov diffeomorphism $f$ of $\mathbb T^n$ such that 
\begin{equation*}
(\lambda_{1,\Leb}(f),\ldots, \lambda_{n,\Leb}(f))=(\xi_1,\ldots,\xi_n)?
\end{equation*}
\end{question}

\begin{question}[Problem 1.3 in \cite{BKRH}, strong flexibility for one measure]\label{strong flexibility 1}
Let $L\colon \mathbb T^n\rightarrow\mathbb T^n$, $n\geq 2$, be a volume-preserving Anosov automorphism with the unstable index $u$. Given any list of numbers 
\begin{equation*}
\xi_1\geq\ldots\geq\xi_u>0>\xi_{u+1}\geq\ldots\xi_n\quad\text{such that:}
\end{equation*}
\begin{equation*}
\sum\limits_{i=1}^n\xi_i=0\qquad\text{and}\qquad \sum\limits_{i=1}^u\xi_i\leq h_{top}(L),
\end{equation*}
does there exist a smooth volume-preserving Anosov diffeomorphism $f\colon\mathbb T^n\rightarrow \mathbb T^n$ homotopic to $L$ such that 
\begin{equation*}
(\lambda_{1,\Leb}(f),\ldots, \lambda_{n,\Leb}(f))=(\xi_1,\ldots,\xi_n)?
\end{equation*}
\end{question}

\begin{question}[Weak flexibility for two measures]\label{weak flexibility 2}
Let $n\in\mathbb N\setminus\{1\}$ and $u\in\mathbb N\cap [1,n)$. Given any two lists of numbers 
\begin{equation*}
\xi_1\geq\ldots\geq\xi_u>0>\xi_{u+1}\geq\ldots\xi_n \quad \text{and} \quad \eta_1\geq\ldots\geq \eta_u>0>\eta_{u+1}\geq\ldots\geq\eta_n \quad\text{such that:}
\end{equation*}
\begin{equation*}
\sum\limits_{i=1}^n\xi_i=0,\qquad \sum\limits_{i=1}^n\eta_i=0,\qquad\text{and}\qquad \sum\limits_{i=1}^u\xi_i\leq\sum\limits_{i=1}^u\eta_i,
\end{equation*}
does there exist a smooth volume-preserving Anosov diffeomorphism $f\colon\mathbb T^n\rightarrow \mathbb T^n$ such that 
\begin{equation*}
(\lambda_{1,\Leb}(f),\ldots, \lambda_{n,\Leb}(f))=(\xi_1,\ldots,\xi_n) \qquad\text{and}\qquad (\lambda_{1,\nu_f}(f),\ldots, \lambda_{n,\nu_f}(f))=(\eta_1,\ldots,\eta_n)?
\end{equation*}
\end{question}

\begin{question}[Strong flexibility for two measures]\label{strong flexibility 2}
Let $L\colon \mathbb T^n\rightarrow\mathbb T^n$, $n\geq 2$, be a volume-preserving Anosov automorphism with the unstable index $u$. Given any two lists of numbers 
\begin{equation*}
\xi_1\geq\ldots\geq\xi_u>0>\xi_{u+1}\geq\ldots\xi_n \quad \text{and} \quad \eta_1\geq\ldots\geq \eta_u>0>\eta_{u+1}\geq\ldots\geq\eta_n \quad\text{such that:}
\end{equation*}
\begin{equation*}
\sum\limits_{i=1}^n\xi_i=0,\qquad \sum\limits_{i=1}^n\eta_i=0,\qquad\text{and}\qquad \sum\limits_{i=1}^u\xi_i<h_{top}(L)<\sum\limits_{i=1}^u\eta_i,
\end{equation*}
does there exist a smooth volume-preserving Anosov diffeomorphism $f\colon\mathbb T^n\rightarrow \mathbb T^n$ homotopic to $L$ such that 
\begin{equation*}
(\lambda_{1,\Leb}(f),\ldots, \lambda_{n,\Leb}(f))=(\xi_1,\ldots,\xi_n) \qquad\text{and}\qquad (\lambda_{1,\nu_f}(f),\ldots, \lambda_{n,\nu_f}(f))=(\eta_1,\ldots,\eta_n)?
\end{equation*}
\end{question}

Corollary 1.6 in \cite{BKRH} gives a positive answer to Question~\ref{weak flexibility 1} when formulated with strict inequalities among the given numbers. For $n=2$, it is folklore among specialists that the positive answer to Question~\ref{strong flexibility 1} was already known by A. Katok using a fairly straightforward global twist construction. The positive answer for $n=2$ and partial answer for $n>2$ follow from \cite[Theorem 1.5]{BKRH}. Moreover, Theorem 1.7 in \cite{BKRH} provides the full solution of Question~\ref{strong flexibility 1} for $\mathbb T^3$ with additional restrictions and the requirement of simple dominated splitting.

In this paper, we study Question~\ref{strong flexibility 2} and provide a positive answer for $n=2$ (see Theorem~\ref{main theorem}). This result can be considered as the two-dimensional version of \cite{E}. All in all, this work differs from \cite{BKRH} by considering flexibility for a pair of exponents instead of a single exponent and by using a more explicit construction. The main difficulty here lies in controlling the measure of maximal entropy. Essentially, the only way to estimate the measures of sets with respect to the measure of maximal entropy is to use Markov partitions which are difficult to understand explicitly for a general Anosov diffeomorphism.

\subsection{Formulation of the result}

Let $\mathbb T^2 = \mathbb R^2/\mathbb Z^2$. Let $f$ be a smooth area-preserving Anosov diffeomorphism on $\mathbb T^2$ homotopic to an Anosov automorphism $L$. We denote by $\lambda_{abs}(f)$ and $\lambda_{mme}(f)$ the positive Lyapunov exponents of $f$ with respect to $\Leb$ and to the measure of maximal entropy $\nu_f$, respectively.

Thus, summarizing Section~\ref{section: anosov} with $n=2$, we have two possibilities:

$$\text{\textit{either} }\qquad 0<\lambda_{abs}(f)<h_{top}(L)<\lambda_{mme}(f)$$
$$\text{\textit{or} }\qquad \lambda_{abs}(f)=\lambda_{mme}(f)=h_{top}(L).$$

Question~\ref{strong flexibility 2} asks if the above relations are the only relations between $\lambda_{abs}$ and $\lambda_{mme}$. Our main theorem shows that this is indeed the case.

\begin{thmintro}\label{main theorem}
Suppose $L_A$ is an Anosov automorphism on $\mathbb T^2$ which is induced by a matrix $A\in SL(2,\Z)$ with $|\text{trace}(A)|>2$. Let $\Lambda=h_{top}(L_A)$. For any $\Lambda_{abs},\Lambda_{mme}\in\mathbb R$ such that $0<\Lambda_{abs}<\Lambda<\Lambda_{mme}$, there exists a smooth area-preserving Anosov diffeomorphism $f\colon \mathbb T^2\rightarrow \mathbb T^2$ homotopic to $L_A$ such that $\lambda_{abs}(f)=\Lambda_{abs}$ and $\lambda_{mme}(f)=\Lambda_{mme}$. 
\end{thmintro}

The shaded area plus its lower right corner in Figure \ref{fig:exponents_set} shows the set of all possible values for pairs $(\lambda_{abs},\lambda_{mme})$ in the setting of Theorem~\ref{main theorem}. 

\begin{figure}[ht]
\centering
\begin{subfigure}{.5\textwidth}
  \centering
  \begin{tikzpicture}
\draw [fill=gray!30, gray!30] (0,2) rectangle (2,6) node[align=center, pos=.5,black] {Region of \\ possible \\ pairs};
\draw[dashed] (0,2) -- (2,2);
\draw[dashed] (2,2) -- (2,6);
\node(x) at (2, 2) {$\bullet$};
\node[align=right](y) at (3, 1) {\footnotesize $(\Lambda,\Lambda)$};
\draw[->] (y) -- (x);
 \draw[->] (-0.5,0) -- (4,0) node at (4,-0.3) {\footnotesize $\lambda_{abs}$}; 
    \draw[->] (0,-0.5) -- (0,6) node[rotate = 90] at (-0.3, 5.5) {\footnotesize $\lambda_{mme}$};
\end{tikzpicture}

  \caption{Possible values of $(\lambda_{abs},\lambda_{mme})$}
  \label{fig:exponents_set}
\end{subfigure}%
\begin{subfigure}{.5\textwidth}
  \centering
  \begin{tikzpicture}
\draw[dashed] (0,2) -- (2,2);
\draw[dashed] (2,2) -- (2,6);
\node[align=right](y) at (3, 1) {\footnotesize  $(\Lambda,\Lambda)$};
\draw[->] (y) -- (x);
\draw[->] (-0.5,0) -- (4,0) node at (4,-0.3) {\footnotesize$\lambda_{abs}$}; 
\draw[->] (0,-0.5) -- (0,6) node[rotate = 90] at (-0.3, 5.5) {\footnotesize $\lambda_{mme}$};
\node[rotate = 90] at (-0.3,3.5) {\small\textbf {Side C}};
\node[rotate = 90] at (2.3,3.5) {\small\textbf {Side B}};
\node at (1, 1.7) {\small\textbf{Side A}};
\draw [fill=gray!30, line width=0.3mm] plot [smooth cycle, tension=0.1] coordinates {(2,2) (1.9,3) (1.95,4) (1.9,5) (1.5, 4.9) (0.09, 4.3) (0.05,4.02) (0.09,3) (0.05,2.1) (1.5, 2.05)(2,2)};
\draw[blue] (0.2, 1.7) -- (0.2, 6);
\draw[blue] (1.8, 1.7) -- (1.8, 6);
\draw[blue] (-0.1, 2.2) -- (2.1, 2.2);
\draw[blue] (-0.1, 3.9) -- (2.1, 3.9);
\draw [line width=0.3mm] plot [smooth] coordinates {(2,2) (1.9,3) (1.95,4) (1.9,5) (1.95,5.7)};
\draw[->,line width=0.5mm] (1.5, 3.5) -- (0.5, 3.5);
\draw[->,line width=0.5mm] (1.5, 2.5) -- (0.5, 2.5);
\node(x) at (2, 2) {$\bullet$};
\end{tikzpicture}

  \caption{Constructions from the proof}
  \label{fig:constr_set}
\end{subfigure}
\caption{}
\label{fig:exponents}
\end{figure}

Using that there is no retraction of a square onto its boundary and continuity of the Lyapunov exponents in the constructed family (see Remark~\ref{remark about continuity}), Theorem~\ref{main theorem} can be reduced to the following.

\begin{thmintro}\label{thm: decrease abs}
Suppose $L_A$ is an Anosov automorphism on $\mathbb T^2$ which is induced by a matrix $A\in SL(2,\Z)$ with $|\text{trace}(A)|>2$. Let $\Lambda=h_{top}(L_A)$. For any positive numbers $\gamma$, $S$ and $T$ such that $\Lambda<S<T$, there exists a smooth family $\{f_{s,t}\}$, where $(s,t)\in[0,1]\times[0,1]$, of area-preserving Anosov diffeomorphisms on $\mathbb T^2$ homotopic to $L_A$ such that the following hold:
\begin{enumerate}
\item $\Lambda-\gamma<\lambda_{abs}(f_{s,0})\leq\Lambda$ for all $s\in[0,1]$;\label{decrease 1}
\item $\lambda_{abs}(f_{s,1})<\gamma$ for all $s\in[0,1]$; \label{decrease 2}
\item $\lambda_{mme}(f_{0,t})<S$ for all $t\in[0,1]$;\label{decrease 3}
\item $\lambda_{mme}(f_{1,t})>T$ for all $t\in[0,1]$.\label{decrease 4}

\end{enumerate}
\end{thmintro}

\begin{remark}\label{remark about continuity}
In a smooth family of Anosov area-preserving diffeomorphisms on $\mathbb T^2$, the Lyapunov exponents with respect to $\Leb$ and the measure of maximal entropy vary continously. For the Lebesgue measure, this follows immediately from \eqref{Lyapunov exponent through integral} and the fact that the unstable distribution $E^u$ varies continuously. In addition, the measure of maximal entropy depends continuously on the dynamics in the weak* topology. To see this, we can use the fact \cite[Theorem 1]{Moser} that for a smooth family of Anosov area-preserving diffeomorphisms, the topological conjugacy to $L_A$ is continuous in the parameters of the family. Moreover, the measure of maximal entropy is mapped to the measure of maximal entropy by the conjugacy. For the families we consider, the continuity of the measure of maximal entropy can alternatively be seen directly using the constructed Markov partition (see Sections~\ref{section: estimate_mme} and~\ref{section: mme in II}). 

\end{remark}

\subsection{Outline of the proof}

To prove Theorem~\ref{thm: decrease abs}, we construct (large) smooth (area-preserving) homotopic deformations of Anosov automorphisms, i.e., deformations preserving the homotopy class and estimate $\lambda_{abs}$ and $\lambda_{mme}$ of the resulting Anosov diffeomorphisms. We refer to Figure \ref{fig:constr_set} in what follows.

\begin{itemize}
\item Without loss of generality, we assume that $\text{trace}(A)>2$. It is enough to prove Theorem~\ref{thm: decrease abs} in that case because of the following. Assume that $B\in SL(2,\Z)$ with $\text{trace}(B)<-2$. Then, $-B\in SL(2,\Z)$, $\text{trace}(-B) > 2$, and $h_{top}(L_B)=h_{top}(L_{-B})=\Lambda$. Assume that $\Lambda_{abs},\Lambda_{mme}\in\mathbb R$ such that $0<\Lambda_{abs}<\Lambda<\Lambda_{mme}$ and there exists an area-preserving Anosov diffeomorphism $f\colon \mathbb T^2\rightarrow\mathbb T^2$ homotopic to $L_{-B}$ such that $\lambda_{abs}(f)=\Lambda_{abs}$ and $\lambda_{mme}(f)=\Lambda_{mme}$. Let $-I = \begin{pmatrix} -1 & 0\\ 0 & -1\end{pmatrix}$. Then, $L_{-I}\circ f\colon \mathbb T^2\rightarrow\mathbb T^2$ is an area-preserving Anosov diffeomorphism homotopic to $L_{B}$ such that  $\lambda_{abs}(L_{-I}\circ f)=\Lambda_{abs}$ and $\lambda_{mme}(L_{-I}\circ f)=\Lambda_{mme}$. 

\item In Section~\ref{section: increase max}, we describe a homotopic deformation that produces for any given positive number $H$ and any small positive number $\gamma$ a curve in the set of possible values of $(\lambda_{abs}, \lambda_{mme})$ $\gamma$-close to Side B with one endpoint at $(\Lambda,\Lambda)$ and the other endpoint corresponding to a smooth area-preserving Anosov diffeomorphims with $\lambda_{mme}$ larger than $H$. The resulting diffeomorphisms have a large twist on a thin strip. In Section~\ref{section: estimate_abs}, we find a lower bound on $\lambda_{abs}$. In Section~\ref{section: estimate_mme}, we provide a lower bound on $\lambda_{mme}$ by controlling the measure of maximal entropy of the constructed diffeomorphisms through Markov partitions. The following theorem summarizes the result in Section~\ref{section: increase max}.
\end{itemize}

\begin{thmintro}(Theorem \ref{thm: increase max})\label{increase max in intro}
Suppose $L_A$ is an Anosov automorphism on $\mathbb T^2$ which is induced by a matrix $A\in SL(2,\Z)$ with $\text{trace}(A)>2$. Let $\Lambda=h_{top}(L_A)$ and $H>0$. For any positive number $\gamma$, there exists a smooth family $\{g_s\}_{s\in[0,1]}$ of area-preserving Anosov diffeomorphisms on $\mathbb T^2$ homotopic to $L_A$ such that $g_0=L_A$ and the following hold:
\begin{enumerate}[label=(\Alph*)]
\item $g_s=L_A$ in a neighborhood of $(0,0)$ for all $s\in[0,1]$;
\item $\Lambda-\gamma<\lambda_{abs}(g_s)\leq \Lambda$ for all $s\in[0,1]$;
\item $\lambda_{mme}(g_1)>H$.
\end{enumerate} 
\end{thmintro}

\begin{itemize}
\item In Section~\ref{section: decrease abs}, starting from the Anosov diffeomorphisms from the first construction with $\lambda_{mme}$ in some range of values, we modify them in a smooth way by a slow-down deformation near a fixed point to get Anosov diffeomorphisms realizing a curve in the set of possible pairs $(\lambda_{abs}, \lambda_{mme})$ arbitrarily close to Side C (see Lemma~\ref{lemma: decreasing abs path}). In this construction, we are able to keep the lower boundary of the realized pairs of exponents arbitrarily close to Side A (see Lemma~\ref{lemma: upper bound on lambda mme start from linear}) and the upper boundary above a line $\lambda_{mme}=T$ (see Lemma~\ref{bound lambda_mme below with variables}), where $T$ depends on the initial range of values of $\lambda_{mme}$ for the diffeomorphisms coming from the first construction.  As a result, we produce a two-parametric family of Anosov diffeomorphisms coming from homotopic deformations covering any given rectangle within the semi-infinite strip that is the set of possible values of $(\lambda_{abs},\lambda_{mme})$. The following theorem summarizes the result in Section~\ref{section: decrease abs}.
\end{itemize}

\begin{thmintro}(Theorem~\ref{decrease abs for family})\label{decrease abs in intro}
Suppose $L_A$ and $\Lambda$ as in Theorem~\ref{increase max in intro}. For any $H$ such that $\Lambda<H$ and positive number $\gamma$, let $\{g_s\}_{s\in[0,1]}$ be a smooth family of area-preserving Anosov diffeomorphisms on $\mathbb T^2$ homotopic to $L_A$ from Theorem~\ref{increase max in intro} applied for $\gamma$ and $H$ with lower bound on $\lambda_{mme}(g_1)$ coming from Lemma~\ref{lemma: bound on mme in I} being larger than $H$. Then, there exists a constant $\tilde C$ such that for any $\sigma>0, S>\Lambda$ there exist 
a smooth family $\{f_{s,t}\}_{(s,t)\in[0,1]\times[0,1]}$ of Anosov diffeomorphisms on $\mathbb T^2$ homotopic to $L_A$ such that:
\begin{enumerate}[label=(\Alph*)]
\item $f_{s,0}=g_s$ for all $s\in[0,1]$;
\item $f_{s,t}$ preserves a probability measure $\mu_{s,t}$ which is absolutely continuous with respect to the Lebesgue measure;
\item $\lambda_{abs}(f_{s,1})<\gamma$ for all $s\in[0,1]$;
\item $\lambda_{mme}(f_{0,t})<S$ for all $t\in[0,1]$;
\item $\lambda_{mme}(f_{1,t})\geq H+\tilde C\sigma$.
\end{enumerate}
\end{thmintro}

\begin{itemize}
\item Let $\{f_{s,t}\}_{(s,t)\in[0,1]\times[0,1]}$ be the family of Anosov diffeomorphisms in Theorem~\ref{decrease abs in intro}. By the Dacorogna-Moser theorem \cite[Theorem, Appendix A]{HJJ}, there exists a $C^{\infty}$ family $\{\Psi_{s,t}\}$ of $C^{\infty}$ diffeomorphisms of $\mathbb T^2$ satisfying $\Psi_{s,t}^*\mu_{s,t} = \Leb$. Let $\tilde{f}_{s,t} = \Psi_{s,t} f_{s,t}\Psi_{s,t}^{-1}$, then $\left\{\tilde{f}_{s,t}\right\}$ is a $C^\infty$ family of Anosov diffeomorphisms on $\mathbb T^2$ that preserve $\Leb$. Since the conjugacy is smooth, we obtain $\lambda_{abs}(\tilde{f}_{s,t})=\lambda_{abs}(f_{s,t})$ and $\lambda_{mme}(\tilde{f}_{s,t})=\lambda_{mme}(f_{s,t})$. Therefore, we obtain Theorem~\ref{thm: decrease abs} if we applied Theorem~\ref{decrease abs in intro} for $H=2T$ and sufficiently small $\sigma$.

\end{itemize}

\subsection{Further questions for Anosov diffeomorphisms of $\mathbb T^2$}

Interestingly, Theorem~\ref{main theorem} can be reformulated as a statement on flexibility for the pressure function among smooth area-preserving Anosov diffeomorphism homotopic to a fixed Anosov automorphism as follows.

Let $\phi^f_t(x) = -t\log\left|Df|_{E_u(x)}\right|$ for any $x\in\mathbb T^2$. This is called the \textit{geometric potential}. The \textit{pressure function} for the potential $\phi_t^f$ is defined by
\begin{equation*}
P(\phi^f_t) = \sup\limits_{\mu}\left(h_\mu(f)+\int_{\mathbb T^2}\phi^f_t\,d\mu\right),
\end{equation*}
where the supremum is taken over all $f$-invariant probability measures on $\mathbb T^2$ and $h_{\mu}(f)$ denotes the measure theoretical entropy of $f$ with respect to $\mu$. It is known that $P(\phi^f_0) = h_{top}(f)$ and $P(\phi^f_1) = 0$. Also, $P(\phi^f_t)$ is a convex real analytic function of $t$ (see for example \cite[Sections 0.2 and 4.6]{RuelleThF} and \cite{BG}). Since $\int_{\mathbb T^2}\phi^f_t\,d\mu$ becomes a dominated term as $t$ tends to $\pm \infty$, we have that $P(\phi^f_t)$ has asymptotes as $t\rightarrow\pm\infty$. Moreover, $\frac{d}{dt}P(\phi^f_t)|_{t=0} = -\lambda_{mme}(f)$ and $\frac{d}{dt}P(\phi^f_t)|_{t=1} = -\lambda_{abs}(f)$.
Thus, Theorem~\ref{main theorem} shows that we can vary the derivatives of the pressure function at $t=0$ and $t=1$. As a result, a more general flexibility question can be formulated in the setting of Theorem~\ref{main theorem}.

\begin{question}
Let $L$ be an Anosov automorphism on $\mathbb T^2$. Given a strictly convex real analytic function $F\colon \mathbb R\rightarrow\mathbb R$ such that $F(0)=h_{top}(L)$, $F(1)=0$, $\frac{dF}{dt}|_{t=0}<-h_{top}(L)$, $\frac{dF}{dt}|_{t=1}\in(-h_{top}(L),0)$, and $F(t)$ has asymptotes as $t\rightarrow\pm\infty$. Does there exist a smooth area-preserving Anosov diffeomorphism $f$ homotopic to $L$ such that $P(\phi^f_t) = F(t)$?
\end{question}

The answer to the above question will require different techniques than presented in this paper. If the answer is negative then it would be interesting to determine which extra conditions on the function must be satisfied. For example, do the higher derivatives of the pressure function \cite{KotaniSunada} provide any additional restrictions? Is there a finite list of conditions that must be added in order to obtain flexibility?

We can also consider a rigidity problem connected to the pressure function. Let $f$ be a smooth area-preserving Anosov diffeomorphism homotopic to an Anosov automorphism $L$. By work of de la Llave \cite{delaLlave} and of Marco and Moriy\'{o}n \cite{MM}, we have that if $\lambda_{abs}(f)=h_{top}(L)$, then $f$ and $L$ are $C^\infty$ conjugate. By the properties of the pressure functions discussed above, we have that if $P(\phi^f_t) = P(\phi^L_t)$ for all $t\in\mathbb R$, then $f$ and $L$ are $C^\infty$ conjugate. A natural question is if $L$ can be replaced by any smooth area-preserving Anosov diffeomorphism.

\begin{question}
Let $f$ and $g$ be smooth area-preserving Anosov diffeomorphisms on $\mathbb T^2$ that are homotopic. Assume $P(\phi^f_t) = P(\phi^g_t)$ for all $t\in\mathbb R$. Does this imply that $f$ and $g$ are $C^\infty$ conjugate?
\end{question}

\vspace{0.5cm}

\textbf{Acknowledgements.} The author would like to thank Andrey Gogolev for helpful discussions and Vaughn Climenhaga for pointing out some references. Question 6 was originally posed by Federico Rodriguez Hertz. The author is also grateful to Anatole Katok for suggesting the problem. Much of this work was completed while the author was at The Ohio State University. She is also grateful for NSF grant 1642548, which supported travel to present the result in this paper. Additionally, the author was partially supported by NSF grant 1547145. The author would like to thank the anonymous referee for thoroughly reading the initial version of the paper and for  suggestions to improve the presentation.

\section{Construction I}\label{section: increase max}

In this section, we prove Theorem~\ref{thm: increase max} using several lemmas. We begin by showing how to deduce the theorem from these lemmas before stating and proving the lemmas themselves.

\begin{theorem}\label{thm: increase max}
Suppose $L_A$ is an Anosov automorphism on $\mathbb T^2$ which is induced by a matrix $A\in SL(2,\Z)$ with $\text{trace}(A)>2$. Let $\Lambda=h_{top}(L_A)$ and $H>0$. For any positive number $\gamma$, there exists a smooth family $\{g_s\}_{s\in[0,1]}$ of area-preserving Anosov diffeomorphisms on $\mathbb T^2$ homotopic to $L_A$ such that $g_0=L_A$ and the following hold:
\begin{enumerate}[label=(\Alph*)]
\item $g_s=L_A$ in a neighborhood of $(0,0)$ for all $s\in[0,1]$;
\item $\Lambda-\gamma<\lambda_{abs}(g_s)\leq \Lambda$ for all $s\in[0,1]$;
\item $\lambda_{mme}(g_1)>H$.
\end{enumerate} 
\end{theorem}

\begin{proof}
Fix $m\in (0,1)$. Let $l\in(0,l_\gamma)$ where $l_\gamma$ comes from Lemma~\ref{lemma: bound on abs in I} applied for $\eps=\gamma$. Let $\beta=\beta_0$ from Lemma~\ref{lemma: bound on mme in I}. Choose $\tilde\delta$ small enough such that $Q\log\mu^+\left(\frac{\beta_0l+\tilde\delta(1-\beta_0)}{\tilde\delta}\right)+(1-Q)\log C> H$, where $Q,C,\mu^+(\cdot)$ are as in Lemma~\ref{lemma: bound on mme in I}. Note that this is possible because $\mu^+(t)\rightarrow\infty$ as $t\rightarrow\infty$. Also, let $w\in (0,w_0(\tilde\delta))$. Then, the family of maps $F^w_{l,\delta}$ defined in \eqref{definition_twist_maps} where $\delta$ varies in $[\tilde\delta, l]$ is the desired family of maps.
\end{proof}



\subsection{Construction I: Anosov twist diffeomorphisms}\label{construction family}

Here we give an explicit formula for a family of smooth area-preserving twist diffeomorphisms \eqref{definition_twist_maps} that are Anosov for an appropriate choice of parameters (see Lemma~\ref{Anosov diffeomorphism}). A smooth subfamily of these diffeomorphisms is used to prove Theorem~\ref{thm: increase max}.

Suppose $L_A$ is an Anosov automorphism on $\mathbb T^2$ which is induced by a matrix $A\in SL(2,\mathbb Z)$ such that $A = \begin{pmatrix} a & b\\ c & d \end{pmatrix}$ and $a+d>2$.

Fix $m\in(0,1)$, let $l\in(0,1-m)$ and $\delta\in(0,l)$. Choose $\beta\in(0,1)$ such that $a+d-|b|\beta>2$. Notice that such $\beta$ exists because $a+d>2$.

Define a function $f_{l,\delta}\colon [0,1]\rightarrow [0,1]$ (see Figure~\ref{f_ldelta}) in the following way:

\begin{equation}\label{function_twist}
f_{l,\delta}(x) = \left\{
\begin{aligned}
&(1-\beta)x+\beta m &\text{ if }\quad &m<x\leq m+l-\delta,\\
&\frac{\beta l+\delta(1-\beta)}{\delta}x-\frac{\beta(m+l)(l-\delta)}{\delta} &\text{ if } \quad&m+l-\delta<x\leq m+l,\\
&x &&\text{ otherwise. }
\end{aligned}
\right .
\end{equation}

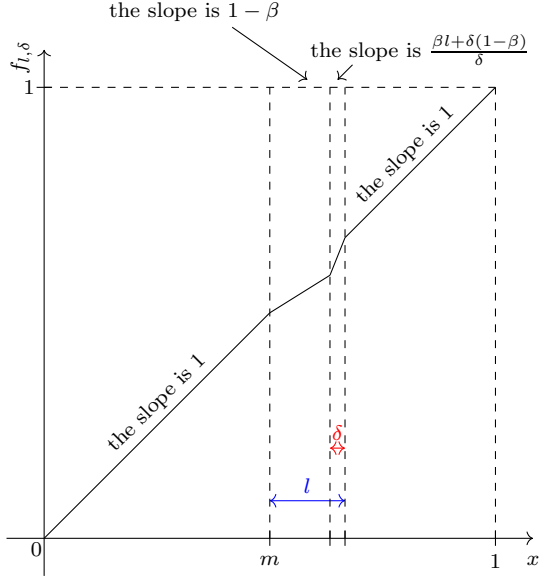
\begin{figure}[ht]
\centering
 \begin{tikzpicture}
\node at (-0.1,-0.15) {\footnotesize $0$};
\draw (6,-0.1) -- (6, 0.1) node at (6,-0.3) {\footnotesize $1$};
\draw (-0.1,6) -- (0.1, 6) node at (-0.2,6) {\footnotesize $1$};
\draw (3,-0.1) -- (3, 0.1) node at (3,-0.3) {\footnotesize $m$};
\draw (3.8,-0.1) -- (3.8, 0.1);
\draw (4,-0.1) -- (4, 0.1);

\draw[<->,blue] (3,0.5) -- (4,0.5) node at (3.5,0.7) {\footnotesize $l$}; 

\draw[<->,red] (3.8,1.2) -- (4,1.2) node at (3.9,1.4) {\footnotesize $\delta$}; 
\draw[dashed] (3,0) -- (3,6);
\draw[dashed] (3.8,0) -- (3.8,6);
\draw[dashed] (4,0) -- (4,6);
\draw[dashed] (6,0) -- (6,6);
\draw[dashed] (0,6) -- (6,6);

\draw (0,0) -- (3,3) node[rotate = 45] at (1.5,1.8) {\footnotesize the slope is $1$};
\draw (4,4) -- (6,6) node[rotate = 45] at (4.8,5.1) {\footnotesize the slope is $1$};
\draw (3,3) -- (3.8,3.5) node at (2,7) {\footnotesize the slope is $1-\beta$};
\draw (3.8, 3.5) -- (4,4)node at (5,6.5) {\footnotesize the slope is $\frac{\beta l+\delta(1-\beta)}{\delta}$};
\draw[->] (2.5,6.8) -- (3.5,6.1);
\draw[->] (4.1,6.3) -- (3.9,6.1);

 \draw[->] (-0.5,0) -- (6.5,0) node at (6.5,-0.3) {\footnotesize $x$}; 
    \draw[->] (0,-0.5) -- (0,6.5) node[rotate = 90] at (-0.3, 6.4) {\footnotesize $f_{l,\delta}$};
\end{tikzpicture}

  \caption{The graph of $f_{l,\delta}$.}
  \label{f_ldelta}
	\end{figure}

The following lemma allows us to obtain a smooth function that coincides with the given continuous piecewise linear function outside small neighborhoods of points of non-smoothness.

\begin{lemma}\label{twist_function}
Let $f_{l,\delta}(x)$ be as in \eqref{function_twist} and $w\in(0,\frac{\delta}{4})$. Denote by $\hat f_{l,\delta}(x)$ the continuous map on $\mathbb R$ that is a lift of $f_{l,\delta}(x)$ such that $\hat f_{l,\delta}(0)=f_{l,\delta}(0)$. Let $\theta_w(x)$ be a smooth positive even function on $\mathbb R$ such that $\int_{\mathbb R}\theta_w(y)dy=1$ and $\theta_w(x)=0$ if $x\not\in(-w,w)$, where $w>0$ and sufficiently small. Define 
\begin{equation}\label{definition of f_w_l_delta}
\hat f^w_{l,\delta}(x)=\int_{\mathbb R}\hat f_{l,\delta}(x-y)\theta_w(y)dy \text{ for any } x\in\mathbb R \qquad \text{ and }\qquad  f^w_{l,\delta}(x)=\hat f^w_{l,\delta}(x) \pmod 1.  
\end{equation}
Then, we have $f^w_{l,\delta}$ is a smooth function on $\mathbb R/\mathbb Z$. Moreover, $f^w_{l,\delta}(x)=f_{l,\delta}(x)$ outside of $w$-neighborhoods of the points $x=m, \, m+l-\delta, \, m+l$. 

In particular, \begin{equation}\label{derivative_twist}
\begin{aligned}
&1-\beta\leq D_xf^w_{l,\delta}\leq 1 &\text{ if }\quad & m-w<x\leq m+w,\\
&D_xf^w_{l,\delta}=1-\beta &\text{ if }\quad & m+w<x\leq m+l-\delta-w,\\
&1-\beta\leq D_xf^w_{l,\delta}\leq \frac{\beta l+\delta(1-\beta)}{\delta} &\text{ if }\quad & m+l-\delta-w<x\leq m+l-\delta+w,\\
&D_xf^w_{l,\delta}=\frac{\beta l+\delta(1-\beta)}{\delta} &\text{ if } \quad& m+l-\delta+w<x\leq m+l-w,\\
&1\leq D_xf^w_{l,\delta}\leq \frac{\beta l+\delta(1-\beta)}{\delta} &\text{ if }\quad &m+l-w<x\leq m+l+w,\\
&D_xf^w_{l,\delta}=1 &&\text{ otherwise. }
\end{aligned}
\end{equation}
\end{lemma}

\begin{proof}
See the proof of Lemma 3.1 in \cite{E}.
\end{proof}

We define $f^0_{l,\delta}=f_{l,\delta}$. For any sufficiently small $w\geq 0$, we consider a family of maps $F^w_{l,\delta}\colon \mathbb T^2\rightarrow \mathbb T^2$, where 
\begin{equation}\label{definition_twist_maps}
F^w_{l,\delta}(x,y) = \begin{pmatrix}
    (a-|b|)x+by+|b|f^w_{l,\delta}(x)\\(c-\sgn(b)d)x+dy+\sgn(b)df^w_{l,\delta}(x)\end{pmatrix} \pmod 1.
\end{equation}

In particular, we have $F^w_{l,\delta} = \begin{pmatrix} a & b\\ c & d\end{pmatrix}$ on $\left\{(x,y)\in\mathbb T^2| x\in[0,m-w]\cup[m+l+w,1)\right\}$.

By the choice of $\beta$, Lemma~\ref{twist_function}, and Lemma~\ref{Anosov diffeomorphism} below, we have that for all $l\in(0,1-m)$, $\delta\in(0,l)$, and sufficiently small $w>0$ the map $F^w_{l,\delta}$ is an area-preserving Anosov diffeomorphism homotopic to $L_A$.

We will need the following general facts.

\begin{proposition}\label{eigenvector_eigenvalue_proposition}
Let 
\begin{equation}\label{definition of A(t)}
A(t) = \begin{pmatrix} a+|b|(t-1) & b\\ c+\sgn(b)d(t-1) & d \end{pmatrix}
\end{equation}
 where $a,b,c,d,t\in\mathbb R$, $ad-bc=1$, and $a+d+(t-1)|b|>2$. Then, $\det(A)=1$ and the eigenvalues of $A(t)$ are
\begin{equation}\label{eigenvalue}
\mu^\pm(t) = \frac{1}{2}\left((a+d+|b|(t-1)\pm\sqrt{\left(a+d+|b|(t-1)\right)^2-4}\right)
\end{equation}
with corresponding eigenvectors
\begin{equation}\label{eigenvector}
    \pmb e^{\pm}(t) = \begin{pmatrix} 2b\\ \phi^\pm\left(a+d+|b|(t-1)\right) \end{pmatrix},
\end{equation}
where
$\phi^{\pm}(u) = 2d-u\pm\sqrt{u^2-4}$ . In particular, $\mu^+(t)>1$ and $0<\mu^-(t)<1$.
\end{proposition}

\begin{proof}
Follows from straightforward computations.
\end{proof}

\begin{lemma}\label{Anosov diffeomorphism}
Suppose $L_A$ is an Anosov area-preserving automorphism on $\mathbb T^2$ which is induced by a matrix $A = \begin{pmatrix} a & b\\ c & d \end{pmatrix}\in SL(2,\mathbb Z)$ with $\text{trace}(A)>2$. Let $f\colon \mathbb R/\mathbb Z \rightarrow \mathbb R/ \mathbb Z$ be a smooth diffeomorphism such that $f(0)=0$, $0<\alpha_1\leq D_xf$ for all $x\in\mathbb R/\mathbb Z$, and $a+d+(\alpha_1-1)|b|>2$. Define a map $F\colon \mathbb T^2\rightarrow\mathbb T^2$ in the following way:
\begin{equation*}
    F(x,y) = \begin{pmatrix}
    (a-|b|)x+by+|b|f(x)\\(c-\sgn(b)d)x+dy+\sgn(b)df(x)\end{pmatrix} \pmod 1.
\end{equation*}
Then, $F$ is a smooth area-preserving Anosov diffeomorphism on $\mathbb T^2$.

Moreover, let $\tilde\alpha>0$ such that $\tilde\alpha<\alpha_1$ and $a+d+(\tilde\alpha-1)|b|>2$. Define the following vectors:
\begin{align}\label{bounds_of_cones}
&\pmb v^+_{min}(t)  = \begin{pmatrix} 2b\\ \phi^+\left(a+d+|b|t\right) \end{pmatrix}, \qquad \pmb v^+_{max} = \begin{pmatrix}
b\\ d
\end{pmatrix}, \\ 
&\pmb v^-_{min}(t) = \begin{pmatrix}
2b\\ \phi^-\left(a+d+|b|t\right)\end{pmatrix},\qquad \text{ and }\qquad \pmb v^-_{max} = \begin{pmatrix}
0\\-1
\end{pmatrix},\nonumber
\end{align}
where $\phi^\pm$ as in Proposition~\ref{eigenvector_eigenvalue_proposition} and $t>\frac{2-(a+d)}{|b|}$.
Let $\mathcal C^\pm$ be the union of the positive cone spanned by $\pmb v^\pm_{min}(\tilde\alpha-1)$ and $\pmb v^\pm_{max}$  and its symmetric complement with vertex at $(0,0)$ in $\mathbb R^2$, respectively. Then, the cones $\mathcal C^\pm$ in $T_{(x,y)}\mathbb T^2$ for all $(x,y)\in\mathbb T^2$ define a system of invariant cones for $F$. Also, there exist $\mu>1>\nu>0$ that depend only on the entries of $A$ and $\tilde\alpha$ such that for all $\pmb v\in\mathcal C^+$ we have $\|D_{(x,y)}F\pmb v\|\geq \mu\|\pmb v\|$ and for all $\pmb v\in\mathcal C^-$ we have $\|D_{(x,y)}F^{-1}\pmb v\|\geq \nu^{-1}\|\pmb v\|$ for any $(x,y)\in\mathbb T^2$.
\end{lemma}

\begin{remark}
We apply Lemma~\ref{Anosov diffeomorphism} for the family of functions $f^w_{l,\delta}$ (see \eqref{definition of f_w_l_delta}) which satisfy $f^w_{l,\delta}(0)=0$ and $D_xf^w_{l,\delta}\geq 1-\beta$ for all $x\in\mathbb R/\mathbb Z$ (see \eqref{derivative_twist}). Recall that we only work with $\beta\in(0,1)$ such that $a+d-|b|\beta>2$. 
\end{remark}

\begin{figure}[ht]
\centering
\begin{tikzpicture}
\draw (0,1) -- (2,1) -- (2,3) -- (0,3) -- (0,1) node at (1,2) {$\mathbb T^2$};
\draw[->] (3,2) -- (4,2) node at (3.5,2.3) {$F$};

\draw[dashed] (5,0) -- (11,0) -- (11,4) -- (5,4) -- (5,0);
\draw[dashed] (5,2) -- (11,2);
\draw[dashed] (7,0) -- (7,4);
\draw[dashed] (9,0) -- (9,4);

\draw (7,1) -- (5,0) -- (7,2) -- (9,3);
\draw (10,3.5) -- (11,4) -- (9,2) -- (8,1.5);
\draw (9.45,3.075) -- (9,3);
\draw (7.45,1.075) -- (7,1);
\draw (9.45,3.075) -- (10,3.5);
\draw (7.45,1.075) -- (8,1.5);
\node at (8.5,2.2) {$F(\mathbb T^2)$};
\end{tikzpicture}
\caption{The image of $\mathbb T^2$ under $F$ obtained from $A=\usebox{\smlmat}$ and $f = f^w_{\frac{1}{4},\frac{1}{16}}$ with parameters $w>0$ and $m=\frac{1}{2}$.}
\end{figure}

\begin{proof}
Notice that $F$ is a smooth map from $\mathbb T^2$ to $\mathbb T^2$ with Jacobian equal to $1$, so it is an area-preserving diffeomorphism.

We show that $F$ is Anosov using invariant cones \cite[Corollary 6.4.8]{KatokHasselblatt}. 

Define $A(t) = \begin{pmatrix} a+|b|(t-1) & b\\ c+\sgn(b)d(t-1) & d \end{pmatrix}$. In particular, we have $D_{(x,y)}F = A(D_xf)$ for any $(x,y)\in\mathbb T^2$.

We point out some properties of $\phi^\pm$. First, we have that $\phi^+(u)$ is monotonically increasing and $\phi^-(u)$ is monotonically decreasing for $u>2$. In particular, $2(d-1)\leq\phi^+(u)<2d$ for $u>2$. Also, both functions are smooth for $u>2$. Moreover, $\lim\limits_{u\rightarrow +\infty}\phi^+(u)=2d$ and $\lim\limits_{u\rightarrow +\infty}\phi^-(u)=-\infty$. Therefore, we obtain that $\mathcal C^+\cap\mathcal C^- = \emptyset$. Moreover, for any $t\geq\alpha_1>\tilde\alpha$, we have  
\begin{equation}\label{invariant cone}
    A(t)\mathcal C^+\subset \Int\, (\mathcal C^+) \qquad \text{and} \qquad A(t)^{-1}\mathcal C^-\subset \Int\, (\mathcal C^-),
\end{equation}
where $\Int$ stands for the interior of the set. See Figure~\ref{figure: cones} for an example of $\mathcal C^\pm$.

We consider the inner product of $\pmb v^+_{min}(\tilde\alpha-1)$ and $\pmb v^+_{max}$
\begin{align}\label{dot product minmax}
\langle\pmb v^+_{min}(\tilde\alpha-1), \pmb v ^+_{max}\rangle =  2b^2+d\phi^+(a+d+|b|(\tilde\alpha-1))\geq \left\{\begin{aligned}&2b^2+2d(d-1) &\text { if } d>0,\\ &2b^2+2d^2 &\text{ if } d\leq 0.\end{aligned}\right.
\end{align}
By \eqref{dot product minmax}, using the fact that $d\in\mathbb Z$, we have that $\langle\pmb v^+_{min}(\tilde\alpha-1), \pmb v ^+_{max}\rangle >0$.

Recall that $\pmb v^+_{min}(\tilde\alpha-1)$ is an eigenvector of $A(\tilde\alpha)$ with an eigenvalue $\mu^+(\tilde\alpha)>1$. Therefore, 
\begin{align*}
\|A(t)\pmb v^+_{min}(\tilde\alpha-1)\|^2 &= \left\|\left[A(\tilde\alpha)+\begin{pmatrix}|b|(t-\tilde\alpha) &0\\ \sgn(b)d(t-\tilde\alpha) &0 \end{pmatrix}\right]\pmb v^+_{min}(\tilde\alpha-1)\right\|^2 \\&= \left\|A(\tilde\alpha)\pmb v^+_{min}(\tilde\alpha-1)+2(t-\tilde\alpha)b\begin{pmatrix}|b|\\\sgn(b)d\end{pmatrix}\right\|^2 \\&= \left\|\mu^+(\tilde\alpha)\pmb v^+_{min}(\tilde\alpha-1)+2|b|(t-\tilde\alpha)\pmb v^+_{max}\right\|^2\\&= \left(\mu^+(\tilde\alpha)\|\pmb v^+_{min}(\tilde\alpha-1)\|\right)^2+4\mu^+(\tilde\alpha)(t-\tilde\alpha)|b|\langle \pmb v^+_{min}(\tilde\alpha-1), \pmb v^+_{max}\rangle+\left(2(t-\tilde\alpha)|b|\|\pmb v^+_{max}\|\right)^2.    
\end{align*}

Therefore, since $\mu^+(\tilde\alpha)>1$ and $\langle\pmb v^+_{min}(\tilde\alpha-1), \pmb v ^+_{max}\rangle >0$, for any $t>\tilde\alpha$ we have $\|A(t)\pmb v^+_{min}(\tilde\alpha-1)\|\geq \mu^+(\tilde\alpha)\|\pmb v^+_{min}(\tilde\alpha-1)\|$.

Moreover, using that $ad-bc=1$, $d\in\mathbb Z$, and $a+|b|(t-1)+d>2$ for any $t>\tilde\alpha$, we have
\begin{align*}
\|A(t)\pmb v^+_{max}\|^2 &= \left\|\begin{pmatrix} a+|b|(t-1) & b\\ c+\sgn(b)d(t-1) & d \end{pmatrix}\begin{pmatrix}b\\d\end{pmatrix}\right\|^2 = \left\|\begin{pmatrix}b(a+|b|(t-1)+d)\\bc+|b|d(t-1)+d^2\end{pmatrix}\right\|^2 \\ &= \left\|\begin{pmatrix}b(a+|b|(t-1)+d)\\-1+d(a+|b|(t-1)+d)\end{pmatrix}\right\|^2\\&= b^2\left(a+d+|b|(t-1)\right)^2+\left(d\left(a+d+|b|(t-1)\right)-1\right)^2\geq 4b^2+(2d-1)^2 \\&= \left(1+\frac{3b^2+3d^2-4d+1}{b^2+d^2}\right)\|\pmb v^+_{max}\|^2.
\end{align*}
In particular, $\|A(t)\pmb v^+_{max}\|\geq \mu^+_{max}\|\pmb v^+_{max}\|$, where $\mu^+_{max} = (1+\frac{3b^2+3d^2-4d+1}{b^2+d^2})>1$ because $3d^2-4d+1\geq 0$ for $d\in\mathbb Z$ and $b\neq 0$ as $a,d\in\mathbb Z$, $ad-bc=1$, and $\text{trace}(A)>2$. 

Let $\mu = \min\{\mu^+(\tilde\alpha), \mu^+_{max}\}>1$. By the properties of $\phi^+$ and the fact that $\text{trace}(A(t))>2$ for $t>\tilde\alpha$, we have that the expanding eigenvectors of $A(t)$ for any $t>\tilde\alpha$ belong to $\mathcal C^+$ (see \eqref{eigenvector}). Moreover, for any $t>\tilde\alpha$, if the set of all vectors in $\mathbb R^2$ that expand at least $\mu$ times by $A(t)$ is non-empty, then it is a cone containing an expanding eigenvector of $A(t)$. As shown above, $\pmb v^+_{min}(\tilde\alpha-1)$ and $\pmb v^+_{max}$ expand at least $\mu$ times by $A(t)$ for any $t>\tilde\alpha$. Since $\mathcal C^+$ is the cone spanned by  $\pmb v^+_{min}(\tilde\alpha-1)$ and $\pmb v^+_{max}$, we have $\|A(t)\pmb v\|\geq \mu\|\pmb v\|$ for all $\pmb v\in\mathcal C^+$ and for all $t>\tilde\alpha$.

Similarly, it can be shown that there exists $\nu\in(0,1)$ such that for all $\pmb v\in\mathcal C^-$, we have $\|A(t)^{-1}\pmb v\|\geq\nu^{-1}\|\pmb v\|$ for all $t>\tilde\alpha$.

Since $D_{(x,y)}F = A(D_xf)$ for any $(x,y)\in\mathbb T^2$ and $D_xf\geq\alpha_1>\tilde\alpha$, by the criterion for map to be Anosov using invariant cones, we obtain that $F$ is Anosov. 
\end{proof}

\begin{figure}[ht]
\centering
\begin{tikzpicture}
\draw[fill=gray!30, line width=0.3mm, gray!30] (0,0)--(3,3)--(3,1.5)--(0,0);
\draw[fill=gray!30, line width=0.3mm, gray!30] (0,0)--(-3,-3)--(-3,-1.5)--(0,0);

\draw[fill=gray!10, line width=0.3mm, gray!10] (0,0)--(0,-3)--(3,-3)--(0,0);
\draw[fill=gray!10, line width=0.3mm, gray!10] (0,0)--(0,3)--(-3,3)--(0,0);

\node at (-0.1,-0.15) {\footnotesize $0$};
\draw [->] (-3, 0)--(3,0);
\draw [->] (0,-3)--(0,3);
\draw (-3,-3)--(3,3);
\draw (-3,-1.5)--(3,1.5);
\draw (3,-3)--(-3,3);

\draw[line width=0.3mm,->] (0,0)--(1,1) node at (0.4,0.9) {\footnotesize $\pmb v^+_{max}$};
\draw[line width=0.3mm,->] (0,0)--(2,1) node at (2,0.5) {\footnotesize $\pmb v^+_{min}(\tilde\alpha-1)$};
\draw[line width=0.3mm,->] (0,0)--(2,1.236) node at (2.2,1.4) {\footnotesize $\pmb e^+(1)$};
\node at (2.6,2) {\footnotesize $\mathcal C^+$};

\draw[line width=0.3mm,->] (0,0)--(0,-1) node at (-0.4,-0.9) {\footnotesize $\pmb v^-_{max}$};
\draw[line width=0.3mm,->] (0,0)--(2,-2) node at (2.85,-1.8) {\footnotesize $\pmb v^-_{min}(\tilde\alpha-1)$};
\draw[line width=0.3mm,->] (0,0)--(2,-3.236) node at (2.45,-3) {\footnotesize $\pmb e^-(1)$};
\node at (-1,2.7) {\footnotesize $\mathcal C^-$};

\end{tikzpicture}
\caption{The cones $\mathcal C^\pm$ for $F$ obtained from $A=\usebox{\smlmat}$ with $\tilde\alpha=0.5$. Here $\pmb e^\pm(1)$ are eigenvectors of $A$.}
\label{figure: cones}
\end{figure}

\subsection{Estimation of $\lambda_{abs}$ in Construction I}\label{section: estimate_abs}
The goal of this section is to prove the following lemma.
\begin{lemma}\label{lemma: bound on abs in I}
 Consider the smooth area-preserving Anosov diffeomorphisms $F^w_{l,\delta}\colon\mathbb T^2\rightarrow \mathbb T^2$ defined in \eqref{definition_twist_maps} (see also Lemma~\ref{twist_function}) using $A=\begin{pmatrix}a & b\\ c& d\end{pmatrix}\in SL(2,\mathbb Z)$ with $a+d>2$. Then, for any $\eps>0$ there exists $l_\eps = l_{\eps}>0$ such that for any $0<l<l_{\eps}$, any $0<\beta<\frac{a+d-2}{|b|}$, any $\delta\in(0,l)$, and any $w\in(0,\frac{\delta}{4})$, we have  $\lambda_{abs}(F^w_{l,\delta})>\Lambda-\eps$, where $\Lambda = h_{top}(L_A) = \log(\mu^+(1))$ (see \eqref{eigenvalue}).
\end{lemma}
\begin{proof}
By Lemma~\ref{Anosov diffeomorphism}, we have the following. For each point $p\in\mathbb T^2$, let $\mathcal C_p^\pm$ be the union of the positive cone in the tangent space at $p$ spanned by $\pmb v^\pm_{min}(-\beta)$ and $\pmb v^\pm_{max}$ (see \eqref{bounds_of_cones}) and its symmetric complement. Then, $\mathcal C_p^+\cap \mathcal C_p^- = \emptyset$, 
\begin{equation}\label{invariant cone direction}
D_p(F^w_{l,\delta})\mathcal C_p^+\subset \mathcal C_{F^w_{l,\delta}(p)}^+, \,\text{ and } \,D_p\left(F^w_{l,\delta}\right)^{-1}\mathcal C_p^-\subset \mathcal C_{\left(F^w_{l,\delta}\right)^{-1}(p)}^-. 
\end{equation}
Moreover, there exists $\mu>1$ that depends only on $A$ and $\beta$ such that 
\begin{equation}\label{general expansion}
\|D_p(F^w_{l,\delta})\pmb v\|\geq \mu\|\pmb v\| \text{ for any } \pmb v\in\mathcal C_p^+.
\end{equation}

Let $\pmb v^u = \pmb e^+(1)$ and $\pmb v^s = \pmb e^-(1)$ (see \eqref{eigenvector}). Then,

\begin{align}\label{boundaries of cone in basis of eigenvectors}
\pmb v^+_{max} &= c^u_{max}\pmb v^u+ c^s_{max}\pmb v ^s, \text{ where } \begin{pmatrix} c^u_{max}\\c^s_{max} \end{pmatrix} = \frac{1}{4
\sqrt{(a+d)^2-4}}\begin{pmatrix} (a+d)+\sqrt{(a+d)^2-4}\\-(a+d)+\sqrt{(a+d)^2-4}\end{pmatrix},\\
\pmb v^+_{min}(-\beta) &= c^u_{min}\pmb v^u+ c^s_{min}\pmb v ^s, \text{ where } \begin{pmatrix} c^u_{min}\\c^s_{min} \end{pmatrix} = \frac{1}{2
\sqrt{(a+d)^2-4}}\begin{pmatrix} \phi^+(a+d-|b|\beta)-\phi^-(a+d)\\\phi^+(a+d)-\phi^+(a+d-|b|\beta)\end{pmatrix},\nonumber
\end{align}
where $\phi^+, \phi^-$ are as in Proposition~\ref{eigenvector_eigenvalue_proposition}.

Moreover, any $\pmb v\in\mathcal C_p^+$ can be written in the form $\pmb v = \alpha_1\pmb v^+_{max}+\alpha_2\pmb v^+_{min}(-\beta)$, where $\alpha_1\alpha_2\geq 0$. In particular, for any $n\in\mathbb N$, we have $\frac{\|A^n\pmb v\|}{\|\pmb v\|}\geq\min\left\{\frac{\|A^n\pmb v^+_{max}\|}{\|\pmb v^+_{max}\|},\frac{\|A^n\pmb v^+_{min}(-\beta)\|}{\|\pmb v^+_{min}(-\beta)\|}\right\}$ if $\pmb v \neq \pmb 0$. Thus, for any $\pmb v\in\mathcal C_p^+$

\begin{equation*}
\|A^n\pmb v\|\geq e^{\Lambda n} \|\pmb v^u\| \left|\sin\angle(\pmb v^u, \pmb v^s)\right|\min\left\{\frac{c^u_{max}}{\|\pmb v^+_{max}\|},\frac{c^u_{min}}{\|\pmb v^+_{min}(-\beta)\|}\right\} \|\pmb v\|.
\end{equation*}

Notice that $\frac{c^u_{\max}}{\|\pmb v^+_{\max}\|}$ depends only on $A$. Also, we have
\begin{equation*}
\frac{c^u_{\min}}{\|\pmb v^+_{\min}(-\beta)\|} = \|\pmb v^u\|^{-1}\left(1+2\frac{c^s_{\min}}{c^u_{\min}}\cdot\frac{\|\pmb v^s\|}{\|\pmb v^u\|}\cos\angle(\pmb v^u,\pmb v^s)+\left(\frac{c^s_{\min}}{c^u_{\min}}\cdot\frac{\|\pmb v^s\|}{\|\pmb v^u\|}\right)^2\right)^{-\frac{1}{2}},
\end{equation*}
where
\begin{equation*}
\frac{c^s_{\min}}{c^u_{\min}} = \frac{2\sqrt{(a+d)^2-4}}{\phi^+(a+d-|b|\beta)-\phi^-(a+d)}-1.
\end{equation*}
Since $\phi^+(u)$ is monotonically increasing and $2(d-1)\leq \phi^+(u)<2d$ for $u>2$, we have that 
\begin{equation*}
\frac{2\sqrt{(a+d)^2-4}}{(a+d)+\sqrt{(a+d)^2-4}}-1<\frac{c^s_{\min}}{c^u_{\min}}\leq\frac{2\sqrt{(a+d)^2-4}}{(a+d)-2+\sqrt{(a+d)^2-4}}-1.
\end{equation*}

As a result, there exists a constant $C>0$ that depends only on $A$ such that
\begin{equation}\label{expansion under A}
\|A^n\pmb v\|\geq e^{\Lambda n} C\|\pmb v\|.
\end{equation}

By the Oseledec multiplicative ergodic theorem, we obtain that 
\begin{equation}\label{lambda_abs}
\lambda_{abs}(F^w_{l,\delta}) = \lim\limits_{n\rightarrow \infty}\frac{\log\|D_{p}\left(F^w_{l,\delta}\right)^n\pmb v\|}{n} 
\end{equation}
for almost every $p\in\mathbb T^2$ with respect to the Lebesgue measure and $\pmb v\in\mathcal C^+_p$. 

Let $\mathcal S = \left\{(x,y)\in\mathbb T^2|m-w\leq x\leq m+l+w\right\}$. We have $\mathrm{Leb}(\mathcal S)=l+2w$, where $\mathrm{Leb}$ is the normalized Lebesgue measure. Moreover, we recall that $F^w_{l,\delta}=L_A$ on $\mathbb T^2\setminus\mathcal S$, in particular, $D_xF^w_{l,\delta}=A$ for all $x\in\mathbb T^2\setminus \mathcal S$. Consider $p\in\mathbb T^2$ and a natural number $n$. We write
\[
n = \sum\limits_{j=1}^sn_j,
\]
where $n_1\in\{0\}\cup\mathbb N$ and the numbers $n_j\in\mathbb N$ for $j=1,\ldots,s$ are chosen in the following way:
\begin{enumerate}
    \item The number $n_1$ is the first moment when $\left(F^w_{l,\delta}\right)^{n_1}(p)\in \mathcal S$;
    \item The number $n_2$ is such that the number $n_1+n_2$ is the first moment when $\left(F^w_{l,\delta}\right)^{n_1+n_2}(p)\in \mathbb T^2\setminus \mathcal S$;
    \item The rest of the numbers are defined in the same way. For any $k\in\mathbb N$, the number $n_{2k+1}$ is such that the number $\sum\limits_{j=1}^{2k+1}n_j$ is the first moment when $\left(F^w_{l,\delta}\right)^{\sum\limits_{j=1}^{2k+1}n_j}(p)\in \mathcal S$, and the number $n_{2k+2}$ is such that the number $\sum\limits_{j=1}^{2k+2}n_j$ is the first moment when $\left(F^w_{l,\delta}\right)^{\sum\limits_{j=1}^{2k+2}n_j}(p)\in \mathbb T^2\setminus\mathcal S$.
\end{enumerate}

Let $\pmb v\in\mathcal C_p^+$ and $\|\pmb v\|=1$. Then, we have
\begin{equation*}
    \log\|D_{p}\left(F^w_{l,\delta}\right)^n\pmb v\| =\sum\limits_{j=1}^s\log\|D_{\left(F^w_{l,\delta}\right)^{n_1+n_2+\ldots+n_{j-1}}(p)}\left(F^w_{l,\delta}\right)^{n_j}\pmb v_j\|,
\end{equation*}
where $\pmb v_1=\pmb v$, $\pmb v_2 = \frac{D_{p}\left(F^w_{l,\delta}\right)^{n_1}\pmb v_1}{\|D_{p}\left(F^w_{l,\delta}\right)^{n_1}\pmb v_1\|}$, and $\pmb v_j = \frac{D_{\left(F^w_{l,\delta}\right)^{n_1+n_2+\ldots+n_{j-2}}(p)}\left(F^w_{l,\delta}\right)^{n_{j-1}}\pmb v_{j-1}}{\|D_{\left(F^w_{l,\delta}\right)^{n_1+n_2+\ldots+n_{j-2}}(p)}\left(F^w_{l,\delta}\right)^{n_{j-1}}\pmb v_{j-1}\|}$ for $j=3, \ldots, s.$ In particular, $\|\pmb v_j\|=1$ for $j=1,\ldots, s$. 

Using \eqref{general expansion} and \eqref{expansion under A}, we obtain for $k\in\mathbb N$
\begin{align*}
\|D_{\left(F^w_{l,\delta}\right)^{n_1+n_2+\ldots+n_{j-1}}(p)}\left(F^w_{l,\delta}\right)^{n_j}\pmb v_j\|&\geq e^{\Lambda n_j}C \qquad &\text{if}\qquad &j=2k-1,\\
\|D_{\left(F^w_{l,\delta}\right)^{n_1+n_2+\ldots+n_{j-1}}(p)}\left(F^w_{l,\delta}\right)^{n_j}\pmb v_j\|&\geq \mu^{n_j} \qquad &\text{if}\qquad &j=2k.
\end{align*}

As a result,
\begin{equation*}
    \log\|D_{p}\left(F^w_{l,\delta}\right)^n\pmb v\| \geq \left[\frac{s}{2}\right]\log C+\Lambda\sum\limits_{k=1}^{[\frac{s}{2}]}n_{2k-1}+(\log\mu)\sum\limits_{k=1}^{[\frac{s}{2}]}n_{2k}.
\end{equation*}

Since $F^w_{l,\delta}$ is a smooth Anosov diffeomorphism, by Birkhoff's Ergodic Theorem we obtain, that 
\begin{equation*}
\frac{1}{n}\sum\limits_{k=1}^{[\frac{s}{2}]}n_{2k-1}\rightarrow (1-l-2w) \,\,\text{ and } \,\,\frac{1}{n}\sum\limits_{k=1}^{[\frac{s}{2}]}n_{2k}\rightarrow (l+2w)\,\, \text{ as }\,\, n\rightarrow\infty. 
\end{equation*}
Moreover, each visit to $\mathcal S$ is at least one iterate, so $\limsup\limits_{n\rightarrow\infty}\left(\frac{1}{n}[\frac{s}{2}]\right)\leq (l+2w)$.

Therefore, since $\mu>1$,
\begin{align*}
    \lambda_{abs}(F^w_{l,\delta})&\geq (l+2w)\min\{\log(C\mu),\log\mu\} + \Lambda(1-l-2w)\\ &\geq (l+2w)\min\{\log(C),0\} + \Lambda(1-l-2w)\\&\rightarrow \Lambda \text{ as } l\rightarrow 0 \text{ since } 0<w<\frac{l}{4}.
\end{align*}

\end{proof}

\subsection{Estimation of $\lambda_{mme}$ in Construction I}\label{section: estimate_mme}

The results of this section can be summarized in the following lemma.
\begin{lemma}\label{lemma: bound on mme in I}
Suppose $L_A$ is an Anosov area-preserving automorphism on $\mathbb T^2$ which is induced by a matrix $A=\begin{pmatrix} a & b\\ c& d\end{pmatrix}\in SL(2,\mathbb Z)$ with $\text{trace}(A)>2$. Fix $m\in(0,1)$ and let $l\in(0,1-m)$ (see the setting of Section~\ref{construction family}). There exists $\beta_0\in\left(0,\frac{a+d-2}{|b|}\right)$ such that there exist $C$ and $\delta_0\in(0,l)$ such that there exists $Q>0$ with the following property. Let $\delta\in(0,\delta_0)$. Then, there exists $w_0=w_0(\delta)$ such that for any $w\in(0,w_0]$ a smooth area-preserving Anosov diffeomorphisms $F^w_{l,\delta}\colon\mathbb T^2\rightarrow \mathbb T^2$ defined in \eqref{definition_twist_maps} with the parameter $\beta=\beta_0$ has the following properties:
\begin{enumerate}
\item $F^w_{l,\delta} = A\left(\frac{\beta_0 l+\delta(1-\beta_0)}{\delta}\right)$ in $S_2^{w_0}(\delta)$ where $A$ is defined in \eqref{definition of A(t)} and $S_2^{w_0}(\delta)$ is defined in \eqref{fast strip in smoothed};
\item $\nu_{F^w_{l,\delta}}(S_2^{w_0}(\delta))\geq Q$ where $\nu_{F^w_{l,\delta}}$ is the measure of maximal entropy of $F^w_{l,\delta}$;
\item $\lambda_{mme}(F^w_{l,\delta})\geq Q\log\mu^+\left(\frac{\beta_0l+\delta(1-\beta_0)}{\delta}\right)+(1-Q)\log C$ where $\mu^+(\cdot)$ is defined in \eqref{eigenvalue}.
\end{enumerate} 


\end{lemma}

The key ingredient to estimate $\lambda_{mme}$ is the construction of a Markov partition, which allows us to represent dynamical systems by symbolic systems (see \cite[Section 18.7]{KatokHasselblatt} for more details). We use the Adler-Weiss construction of a Markov partition \cite{AdlerWeiss} to construct a Markov partition of $F^w_{l,\delta}$.

Let $p\in\mathbb T^2$. Then, 
\begin{equation*}
W_w^s(p) = \{z\in\mathbb T^2\,|\, \lim\limits_{n\rightarrow\infty}\dist\left(\left(F^w_{l,\delta}\right)^n(z),\left(F^w_{l,\delta}\right)^n(p)\right)=0\}
\end{equation*}
 and 
\begin{equation*}
W_w^u(p) = \{z\in\mathbb T^2\,|\, \lim\limits_{n\rightarrow\infty}\dist\left(\left(F^w_{l,\delta}\right)^{-n}(z),\left(F^w_{l,\delta}\right)^n(p)\right)=0\}
\end{equation*}
are the stable and unstable manifolds of $F^w_{l,\delta}$ at $p$, respectively. Moreover, for any $\eps>0$ let $W_w^{i}(p,\eps)$ be the $\eps$-neighborhood of $p$ in $W_w^{i}(p)$, where $i=u,s$. We denote $F_{l,\delta}=F^0_{l,\delta}$, and $W^i(p)=W^i_0(p)$ for $i=u,s$. 

Let $\pmb v^u = \pmb e^+(1)$ and $\pmb v^s = \pmb e^-(1)$ (see \eqref{eigenvector}). Since $F^w_{l,\delta}=L_A$ in a neighborhood of $(0,0)$ by the construction of $F^w_{l,\delta}$, there exists $\kappa>0$ such that for any sufficiently small $w\geq 0$, we have 
\begin{equation*}
W_w^{i}((0,0),\kappa) = \left\{(x,y)\,|\, \dist((x,y),(0,0))\leq\kappa, \, \left\langle \begin{pmatrix}-y\\x\end{pmatrix},\pmb v^i\right\rangle=0\right\} \text{ for } i=u,s. 
\end{equation*}
Moreover, since $(0,0)$ is a fixed point for $F^w_{l,\delta}$, we have 
\begin{equation}\label{stable_unstable_mfld_image}
W^s_w((0,0)) = \bigcup\limits_{n\in\mathbb N}\left(F^w_{l,\delta}\right)^{-n}\left(W_w^{s}((0,0),\kappa)\right)\, \text{ and } \, W^u_w((0,0)) = \bigcup\limits_{n\in\mathbb N}\left(F^w_{l,\delta}\right)^{n}\left(W_w^{u}((0,0),\kappa)\right).
\end{equation}

\subsubsection*{Markov partition for $F_{l,\delta}$.}

Let us draw segments of $W^u((0,0))$ and $W^s((0,0))$ until they cross sufficiently many times and separate $\mathbb T^2$ into two disjoint (curvilinear) rectangles $R_1, R_2$. Define $\mathcal R$ to be the partition into rectangles determined by $R_i\cap F_{l,\delta}(R_j)$, where $i,j=1,2$. See Figure~\ref{initial_Markov} for an example. For $n\in\mathbb N$ let $\mathcal R^n$ be the partition into rectangles generated by $\left(F_{l,\delta}\right)^i(R')\cap \left(F_{l,\delta}\right)^j(R'')$, where $R', R''\in\mathcal R$ and $i,j=-n,-n+1,\ldots, n-1,n$. Let $\mathcal R^0 = \mathcal R$. Note that the $\mathcal R^n$ depend on $A$, $\beta$, $m$, $l$, and $\delta$ even though we do not emphasize this in the notation. By the construction, for each $R\in\mathcal R^n$, we have that two opposite boundaries of $R$ are contained in $W^u((0,0))$ and the other two are contained in $W^s((0,0))$.

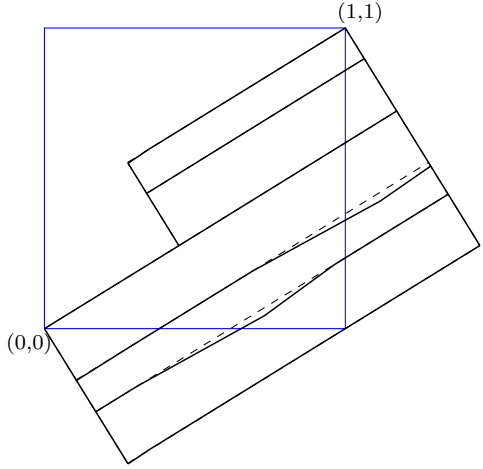
\begin{figure}[ht]
      \centering
			\begin{tikzpicture}[scale=0.4]
			\draw[dashed] (1.708,-2.764) -- (13.42,4.47);
			\draw[dashed] (1.055,-1.707) -- (12.76,5.528);
			\draw[dashed] (2.764,5.528) -- (10,10);
			\draw[dashed] (2.764, -4.48) -- (14.472,2.764);
			\draw[dashed] (2.764, -4.48) -- (0,0);
			\draw[dashed] (2.764,5.528) -- (4.47,2.76);
			\draw[dashed] (0,0)--(11.708,7.236);
			\draw[dashed] (10,10)--(14.472,2.764);
			\draw[dashed] (3.4,4.5) -- (10.63,8.98);
			\draw[blue] (0,0)--(10,0)--(10,10)--(0,10)--(0,0);
			\draw[line width=0.2mm] (0,0)--(4.47,2.76)--(11.71,7.24);
			\draw[line width=0.2mm] (1.065,-1.723) --(2.88,-0.58) --(6.91,1.91)--(11.16,4.242)--(12.839,5.406);
			\draw[line width=0.2mm] (1.708,-2.764)--(3.09,-1.91)--(7.305,0.425)--(9.64,2.13)--(13.42,4.47);
			\draw[line width=0.2mm] (2.78,-4.5)--(3.46,-4.04)--(14.472,2.764);
			\draw[line width=0.2mm] (3.4,4.5) -- (10.63,8.98);
			\draw[line width=0.2mm] (2.78,5.5) -- (3.46,5.96) -- (10,10);
			\draw[line width=0.2mm] (0,0)--(2.78,-4.5);
			\draw[line width=0.2mm] (2.78,5.5)--(4.47,2.76);
			\draw[line width=0.2mm] (10,10) -- (14.472, 2.764);
			\node at (-0.5,-0.5) {\footnotesize (0,0)};
			\node at (10.5,10.5) {\footnotesize (1,1)};
			\end{tikzpicture}
      \caption{The image of the partition $\mathcal R$ (in black, solid) for $F_{l,\delta}$ which is a perturbation of $L_A$ where $A=\usebox{\smlmat}$. Some boundaries are compared to the partition of $L_A$ (in black, dashed). $\mathbb T^2$ is shown in blue.}
      \label{initial_Markov}
\end{figure}

We have the following property for the constructed partition.

\begin{lemma}\label{element inside}
There exists $\beta_0\in(0,\frac{a+d-2}{|b|})$ and there exists $\delta_0(\beta_0)\in(0,l)$ such that there exists $n_0\in\mathbb N$ with the following property. Let $\delta\in(0,\delta_0)$ and $\mathcal R$ be the Markov partition for $F_{l,\delta}$ (described above) with $\beta=\beta_0$. Then, there exists $R\in\mathcal R^{n_0}$ such that $R\subset \{(x,y)\in\mathbb T^2| x\in(m+l-\delta,m+l)\}$.
\end{lemma}

Lemma~\ref{element inside} will follow from the lemmas below. First, we introduce some notation and definitions.

Denote
\begin{align}\label{different regions}
   S_1(\delta)=\{(x,y)\in\mathbb T^2&|x\in[m,m+l-\delta]\}; \; S_2(\delta)=\{(x,y)\in\mathbb T^2|x\in[m+l-\delta,m+l]\}; 
\\
    &S_3=\mathbb T^2\setminus\{(x,y)\in\mathbb T^2|x\in(m,m+l)\}.\nonumber
		\end{align}

\begin{definition}
Let $n\in\mathbb N$ and $\mathcal R^n$ be as described above. Let $R\in\mathcal R^n$. The \underline{$s$-size of $R$}, denoted by $d_s(R)$, is the distance in the $\pmb v^s$ direction between the two opposite boundaries (or their extensions) of $R$ that are contained in $W^u((0,0))$. The \underline{$s$-size of $\mathcal R^n$} is $\tilde d_s(\mathcal R^n) = \max\limits_{R\in\mathcal R^n} d_s(R)$.
\end{definition}

\begin{definition}
Let $n\in\mathbb N$ and $\mathcal R^n$ be as described above. Let $R\in\mathcal R^n$. The \underline{$u$-size of $R$}, denoted by $d_u(R)$, is the distance in the $\pmb v^u$ direction between the two opposite boundaries (or their extensions) of $R$ that are contained in $W^s((0,0))$. The \underline{$u$-size of $\mathcal R^n$} is $\tilde d_u(\mathcal R^n) = \max\limits_{R\in\mathcal R^n} d_u(R)$. 
\end{definition}

\begin{lemma}\label{stable_size}
Consider the setting above. Then, there exists $\beta_s\in(0,\frac{a+d-2}{|b|})$ such that for any $\beta\in(0,\beta_s)$ there exists $\delta_s=\delta_s(\beta)\in(0,l)$ such that there exists a constant $\nu_s\in(0,1)$ with the following properties. Let $\delta\in(0,\delta_s)$ and $\mathcal R$ be the partition for $F_{l,\delta}$. Then, for any $n\in\mathbb N$, and any $R\in\mathcal R^n$ we have
$$d_s(F_{l,\delta}(R))<\nu_s d_s(\mathcal R^n) .$$
\end{lemma}
\begin{proof}
Let $\delta\in(0,l)$ and $\beta\in(0,\frac{a+d-2}{|b|})$. Consider the partition $\mathcal R$ for $F_{l,\delta}$. Let $R\in\mathcal R^n$ and $d_s(R)=r$. Define $p_1, p_2\in W^u((0,0))$ be the points on the opposite boundaries of $F_{l,\delta}(R)$ such that the segment $[p_1,p_2]$ has direction $\pmb v^s$. Partition $[p_1,p_2]$ into the minimal number of segments such that each segment is fully contained in one of the regions $F_{l,\delta}(S_1)$, $F_{l,\delta}(S_2)$, and $F_{l,\delta}(S_3)$.

Recall that $F_{l,\delta}$ is linear in $S_1$, $S_2$, and $S_3$, so it takes a piece of a line to a piece of a line.

First, we find the directions of lines in $S_1$ and $S_2$ that are taken to lines with the direction $\pmb v^s$. Let $A(t)$ be as in Proposition~\ref{eigenvector_eigenvalue_proposition}. By the definition of $F_{l,\delta}$, we have $DF_{l,\delta} = A(1-\beta)$ on $S_1$ and $DF_{l,\delta} = A\left(\frac{\beta l+\delta(1-\beta)}{\delta}\right)$ on $S_2$.

It is easy to see that $A(t)\pmb u(t) = \pmb v^s$ if and only if 
\begin{equation*}
\pmb u (t)= \begin{pmatrix}u_1(t)\\u_2(t) \end{pmatrix}= \left(A(t)\right)^{-1}\pmb v^s = e^\Lambda \pmb v^s+\sgn(b)(t-1)\begin{pmatrix}0\\ -dv^s_1+bv^s_2\end{pmatrix},
\end{equation*}
where $\pmb v^s = \begin{pmatrix}v^s_1\\v^s_2\end{pmatrix}$ and $\Lambda = h_{top}(L_A)$.

As a result, for any $\eps_1>0$ there exists $\beta_1=\beta_1(\eps_1, A)\in\left(0,\frac{a+d-2}{|b|}\right)$ such that for any $\beta\in (0,\beta_1(\eps_1))$ we have $\frac{\|\pmb v^s\|}{\|\pmb u(1-\beta)\|}< e^{-\Lambda}+\eps_1$. Moreover, for any $\eps_2>0$ there exists $\delta_1=\delta_1(\eps_2,A,l,\beta)\in(0,l)$ such that for any $\delta\in(0,\delta_1(\eps_2))$, we have $\frac{\|\pmb v^s\|}{\left\|\pmb u\left(\frac{\beta l +\delta(1-\beta)}{\delta}\right)\right\|}<\eps_2$.

Also,
\begin{align}\label{slope of u(t)}
    \frac{u_2(t)}{u_1(t)} &= \frac{v_2^s}{v_1^s} +\sgn(b)(t-1)e^{-\Lambda}\left(-d+b\frac{v_2^s}{v_1^s}\right) \\&=\frac{v_2^s}{v_1^s} -\sgn(b)(t-1)\frac{1}{4}(d+a-\sqrt{(a+d)^2-4})(d+a+\sqrt{(a+d)^2-4})\nonumber\\&= \frac{v_2^s}{v_1^s} -\sgn(b)(t-1).\nonumber
\end{align}

According to the definition of $F_{l,\delta}$, the construction of $\mathcal R^n$, and \eqref{invariant cone direction}, $W^u((0,0))$ is a piecewise linear curve with linear pieces having directions in the cone spanned by $\pmb v^+_{min}(-\beta)$ and $\pmb v^+_{max}$.

Denote by $L(I)$ the length of $I\subset \mathbb R^2$ in the standard Euclidean metric. We estimate $L([p_1,p_2])$. In the following discussion the phrase ``a side has direction $\pmb v$" means ``a side is contained in a line with direction $\pmb v$".

\underline{Case 1:} Assume $[p_1,p_2]\subset F_{l,\delta}(S_3)$. Then, $\left(F_{l,\delta}\right)^{-1}([p_1,p_2])$ has direction $\pmb v^s$ and $F_{l,\delta} = L_{A}$ on $S_3$, so $L([p_1,p_2])\leq e^{-\Lambda}d_s(\mathcal R^n)$.

\underline{Case 2:} Assume $[p_1,p_2]\subset F_{l,\delta}(S_1)$. 

Let $\Delta D_1D_2D_3$ be a triangle such that the following hold: $D_1D_2$ has direction $\pmb u(1-\beta)$, $D_1D_3$ has direction $\pmb v^s$, $L(D_1D_3)=d_s(\mathcal R^n)$, and $D_2D_3$ has direction $\pmb v^+_{min}(-\beta)$. Then, for sufficiently small $\beta$ which depends only on $A$,

\begin{equation*}
L(D_1D_2)= d_s(\mathcal R^n)\left(1+\frac{2|b|\beta}{\sqrt{(a+d)^2-4}+\sqrt{(a+d-|b|\beta)^2-4}-|b|\beta}\right)\sqrt{\frac{1+\left(\frac{v^s_2}{v^s_1}+\sgn(b)\beta\right)^2}{1+\left(\frac{v^s_2}{v^s_1}\right)^2}}.
\end{equation*}

Let $\Delta T_1T_2T_3$ be a triangle such that the following hold: $T_1T_2$ has direction $\pmb u(1-\beta)$, $T_1T_3$ has direction $\pmb v^s$, $L(T_1T_3)=d_s(\mathcal R^n)$, and $T_2T_3$ has direction $\pmb v^+_{max}$. Then, for sufficiently small $\beta$ which depends only on $A$,

\begin{equation*}
L(T_1T_2)= d_s(\mathcal R^n)\left(1+\frac{2|b|\beta}{(a+d)+\sqrt{(a+d)^2-4}-2|b|\beta}\right)\sqrt{\frac{1+\left(\frac{v^s_2}{v^s_1}+\sgn(b)\beta\right)^2}{1+\left(\frac{v^s_2}{v^s_1}\right)^2}}.
\end{equation*}

Using \eqref{slope of u(t)}, we obtain that $L\left(\left(F_{l,\delta}\right)^{-1}([p_1,p_2])\right)\leq \max\{L(D_1D_2), L(T_1T_2)\}$ and $$L([p_1,p_2])\leq \max\{L(D_1D_2), L(T_1T_2)\}\frac{\|\pmb v^s\|}{\|\pmb u(1-\beta)\|},$$ where $\frac{\max\{L(D_1D_2), L(T_1T_2)\}}{d_s(\mathcal R^n)}\rightarrow 1$ as $\beta\rightarrow 0$. Therefore, there exists $\beta_2 = \beta_2(A)\in\left(0,\frac{a+d-2}{|b|}\right)$ such that there exists $\nu_2=\nu_2(\beta_2)\in(0,1)$ such that for any $\beta\in(0,\beta)$ we have $L([p_1,p_2])\leq \nu_2d_s(\mathcal R^n)$.

\underline{Case 3:} Assume $[p_1,p_2]\subset F_{l,\delta}(S_2)$.

Let $\Delta D_1D_2D_3$ be a triangle such that the following hold: $D_1D_2$ has direction $\pmb u\left(\frac{\beta l +\delta(1-\beta)}{\delta}\right)$, $D_1D_3$ has direction $\pmb v^s$, $L(D_1D_3)=d_s(\mathcal R^n)$, and $D_2D_3$ has direction $\pmb v^+_{min}(-\beta)$. Then, 

\begin{equation*}
L(D_1D_2)= d_s(\mathcal R^n)\left|1-\frac{\frac{\sgn(b)\beta(l-\delta)}{\delta}}{\frac{\sqrt{(a+d)^2-4}+\sqrt{(a+d-|b|\beta)^2-4}+|b|\beta}{2b}+\frac{\sgn(b)\beta(l-\delta)}{\delta}}\right|\sqrt{\frac{1+\left(\frac{v^s_2}{v^s_1}-\frac{\sgn(b)\beta(l-\delta)}{\delta}\right)^2}{1+\left(\frac{v^s_2}{v^s_1}\right)^2}}.
\end{equation*}

Let $\Delta T_1T_2T_3$ be a triangle such that the following hold: $T_1T_2$ has direction $\pmb u\left(\frac{\beta l +\delta(1-\beta)}{\delta}\right)$, $T_1T_3$ has direction $\pmb v^s$, $L(D_1D_3)=d_s(\mathcal R^n)$, and $D_2D_3$ has direction $\pmb v^+_{max}$. Then, 

\begin{equation*}
L(T_1T_2)= d_s(\mathcal R^n)\left|1-\frac{\frac{\sgn(b)\beta(l-\delta)}{\delta}}{\frac{d}{b}-\frac{v^s_2}{v^s_1}+\frac{\sgn(b)\beta(l-\delta)}{\delta}}\right|\sqrt{\frac{1+\left(\frac{v^s_2}{v^s_1}-\frac{\sgn(b)\beta(l-\delta)}{\delta}\right)^2}{1+\left(\frac{v^s_2}{v^s_1}\right)^2}}.
\end{equation*}

Using \eqref{slope of u(t)}, we obtain that $L\left(\left(F_{l,\delta}\right)^{-1}([p_1,p_2])\right)\leq \max\{L(D_1D_2), L(T_1T_2)\}$. In particular, for any fixed $\beta$, there exists $\delta_3=\delta_3(\beta,A,l)\in(0,l)$ such that for any $\delta\in(0,\delta_3)$, we have 
\begin{equation*}
L([p_1,p_2])=L\left(\left(F_{l,\delta}\right)^{-1}([p_1,p_2])\right)\frac{\|\pmb v^s\|}{\left\|\pmb u\left(\frac{\beta l +\delta(1-\beta)}{\delta}\right)\right\|}\leq e^{-\Lambda}d_s(\mathcal R^n).
\end{equation*}

\underline{Case 4:} Assume $[p_1,p_2]\subset F_{l,\delta}(S_1) \cup F_{l,\delta}(S_3)$ and $[p_1,p_2]\cap \Int(F_{l,\delta}(S_i))\neq\emptyset$ for $i=1,3$. Combining Cases 1 and 2, we obtain that there exists $\beta_2\in\left(0,\frac{a+d-2}{|b|}\right)$ and $\nu_2 = \nu_2(\beta_2)\in(0,1)$ such that for any $\beta\in(0,\beta_2)$ we have 
$L([p_1,p_2])\leq \max\{e^{-\Lambda},\nu_2\}d_s(\mathcal R^n).$

\underline{Case 5:} Assume $[p_1,p_2]\subset F_{l,\delta}(S_2) \cup F_{l,\delta}(S_3)$ and $[p_1,p_2]\cap \Int(F_{l,\delta}(S_i))\neq\emptyset$ for $i=2,3$. Combining Cases 1 and 3, we obtain that for any $\beta\in(0,\beta_2)$ (where $\beta_2$ defined in Case 4) we have that there exists $\delta_3=\delta_3(\beta,A,l)\in(0,l)$ such that for any $\delta\in(0,\delta_3)$ we have $L([p_1,p_2])\leq e^{-\Lambda}d_s(\mathcal R^n).$

\underline{Case 6:} Assume $[p_1,p_2]\subset F_{l,\delta}(S_1) \cup F_{l,\delta}(S_2)$ and $[p_1,p_2]\cap \Int(F_{l,\delta}(S_i))\neq\emptyset$ for $i=1,2$. There is a piece of a line with direction $\pmb v^s$ that intersects two opposite boundaries of a rectangle in $\mathcal R^n$ that are contained in $W^u((0,0))$ and passes through a point of $\left(F_{l,\delta}\right)^{-1}([p_1,p_2])$ that belongs to the line $x=m+l-\delta$ (on $\mathbb T^2$). Then, using Cases 2 and 3, we obtain that $L([p_1,p_2])\leq \max\{\nu_2, e^{-\Lambda}\}d_s(\mathcal R^n)$. 

\underline{Case 7:} Assume $[p_1,p_2]\subset F_{l,\delta}(S_1)\cup F_{l,\delta}(S_2)\cup F_{l,\delta}(S_3)$ and $[p_1,p_2]\cap \Int(S_i)\neq \emptyset$ for $i=1,2,3$. Let $q_j=\left(F_{l,\delta}\right)^{-1}(p_j)$ where $j=1, 2$

\underline{Subcase 7.1:} Assume $q_j\in S_3$ where $j=1, 2$.

Recall that since $F_{l,\delta}=L_A$ on $S_3$ and $F_{l,\delta}$ is a piecewise linear map, if two segments $\gamma_1,\gamma_2\subset F_{l,\delta}(S_3)\cap[p_1,p_2]$, then the segments $\left(F_{l,\delta}\right)^{-1}(\gamma_1)$ and $\left(F_{l,\delta}\right)^{-1}(\gamma_2)$ are subsets of a line with direction $\pmb v^s$ connecting the opposite boundaries of a rectangle in $\mathcal R^n$ that are contained in $W^u((0,0))$.

Let $Q_i$ be the maximal segment in $[q_1,q_2]$ such that $Q_i\subset S_i$, where $i=1,2$. We denote $r_1=L(Q_1)$ and $r_2=L(Q_2)$. Let $Q'_i$ be the maximal segment in $\left(F_{l,\delta}\right)^{-1}([p_1,p_2])$ such that $Q'_i\subset S_i$, where $i=1,2$. We denote $r'_1=L(Q'_1)$ and $r'_2=L(Q'_2)$. By the definition of $F_{l,\delta}$ and \eqref{slope of u(t)}, we have
\begin{align*}
    &r_1= (l-\delta)\sqrt{1+\left(\frac{v_2^s}{v_1^s}\right)^2}, \qquad r_2= \delta \sqrt{1+\left(\frac{v_2^s}{v_1^s}\right)^2}, \\
    &r'_1= (l-\delta)\sqrt{1+\left(\frac{v_2^s}{v_1^s}+\sgn(b)\beta\right)^2} = r_1\frac{\sqrt{1+\left(\frac{v_2^s}{v_1^s}+\sgn(b)\beta\right)^2}}{\sqrt{1+\left(\frac{v_2^s}{v_1^s}\right)^2}}, \\
    &r'_2=\delta\sqrt{1+\left(\frac{v_2^s}{v_1^s}-C_\delta\right)^2} = r_2\frac{\sqrt{1+\left(\frac{v_2^s}{v_1^s}-C_\delta\right)^2}}{\sqrt{1+\left(\frac{v_2^s}{v_1^s}\right)^2}},
\end{align*}
where $C_\delta = \sgn(b)\frac{\beta(l-\delta)}{\delta}$. Then,
\begin{equation}\label{max_length_in_s}
L(F_{l,\delta}(Q'_1)) = r'_1\frac{\|\pmb v^s\|}{\|\pmb u(1-\beta)\|} = r_1e^{-\Lambda}\sqrt{\frac{1+\left(\frac{v_2^s}{v_1^s}+\sgn(b)\beta\right)^2}{1+\left(\frac{v_2^s}{v_1^s}-e^\Lambda\beta\sgn(b)\left(-d+b\frac{v_2^s}{v_1^s}\right)\right)^2}}
\end{equation}
and
\begin{equation}\label{max_length_in_f}
L(F_{l,\delta}(Q_2'))=r'_2\frac{\|\pmb v^s\|}{\left\|\pmb u\left(\frac{\beta l+\delta(1-\beta)}{\delta}\right)\right\|} = r_2e^{-\Lambda}\sqrt{\frac{1+\left(\frac{v_2^s}{v_1^s}-C_\delta\right)^2}{1+\left(\frac{v_2^s}{v_1^s}+C_\delta e^{-\Lambda}\left(-d+b\frac{v_2^s}{v_1^s}\right)\right)^2}}.
\end{equation}

Choose $K>0$ such that $e^{-\Lambda}(1+K)<1$ and $\left(-d+b\frac{v_2^s}{v_1^s}\right)^{-1}+K = \frac{2}{a+d+\sqrt{(a+d)^2-4}}+K<1$ which is possible since $\Lambda>0$ and $a+d>2$. Let $\nu_{7} = \max\{e^{-\Lambda}(1+K), \frac{2}{a+d+\sqrt{(a+d)^2-4}}+K\}$. Then, there exists $\beta_{7}=\beta_7(A,K)\in(0,\frac{a+d-2}{|b|})$ such that for any $\beta\in(0,\beta_7)$ there exists $\delta_7 = \delta_7(\beta,K,A,l)\in(0,l)$ such that for any $\delta\in(0,\delta_7)$ we have $L(F_{l,\delta}(Q'_i))\leq r_i\nu_7$ for $i=1,2$ because we have \eqref{max_length_in_s} and \eqref{max_length_in_f}. Combining this with Case 1 and the fact that $L([q_1,q_2])\leq d_s(\mathcal R^n)$, we obtain that for the specified choice of $\beta_7$ and $\delta_7$, we have that $L([p_1,p_2])\leq \nu_7d_s(\mathcal R^n)$ for any $\delta\in(0,\delta_7)$.

\underline{Subcases 7.2-7.3} Let $Q'_i$ be the maximal piece of $\left(F_{l,\delta}\right)^{-1}([p_1,p_2])$ in $S_i$ for $i=1,2,3$. Notice that $Q'_3$ has direction $\pmb v^s$. Let $Q$ be a piece of a line with direction $\pmb v^s$ connecting the opposite boundaries of a rectangle in $\mathcal R^n$ that are contained in $W^u((0,0))$ which passes through a point of $\left(F_{l,\delta}\right)^{-1}([p_1,p_2])$ that belongs to the line $x=m+l-\delta$.

\underline{Subcase 7.2:} Assume $q_1\in S_3$ and $q_2\in S_2$. 

Let $Q_1$ be the maximal piece of $Q$ in $\left\{(x,y)\in\mathbb T^2| 0\leq x\leq m+l-\delta\right\}$ and $Q_2=Q\setminus Q_1$. 

Let $D_1D_2D_3D_4(\pmb z)$ be a trapezoid such that $D_1D_2$ and $D_3D_4$ have direction $\pmb v^s$, $L(D_1D_2)=L(Q_1)$, $D_2D_3$ has direction $\pmb u(1-\beta)$, $L(D_2D_3)=(l-\delta)\sqrt{1+\left(\frac{v_2^s}{v_1^s}+\sgn(b)\beta\right)^2}$, and $D_1D_4(\pmb z)$ has direction $\pmb z = \begin{pmatrix}z_1\\z_2\end{pmatrix}$. In our case, we consider $\pmb z=\pmb v^+_{max}, \pmb v^+_{min}(-\beta)$. Moreover, to have a trapezoid for our choice of $\pmb z$, we should have $\frac{L(Q_1)}{\sqrt{1+\left(\frac{v^s_2}{v^s_1}\right)^2}}\geq(l-\delta)\left(1-\frac{\sgn(b)\beta}{\frac{z_2}{z_1}-\frac{v_2^s}{v_1^s}}\right)$. Then, we obtain 
\begin{equation*}
L(D_2D_3)+L(D_3D_4) = L(Q_1)+(l-\delta)\left(\sqrt{1+\left(\frac{v^s_2}{v^s_1}+\sgn(b)\beta\right)^2}-\left(1-\frac{\sgn(b)\beta}{\frac{z_2}{z_1}-\frac{v_2^s}{v_1^s}}\right)\sqrt{1+\left(\frac{v^s_2}{v^s_1}\right)^2}\right).
\end{equation*}
In particular, for sufficiently small $\beta>0$ (depending on $A$), we have
\begin{equation}\label{length_in_trapezoid_before}
L(D_2D_3)+L(D_3D_4) \leq L(Q_1)\max\left\{1,\max\limits_{\pmb z\in\{\pmb v^+_{max}, \pmb v^+_{min}(-\beta)\}}\left\{\frac{\sqrt{1+\left(\frac{v^s_2}{v^s_1}+\sgn(b)\beta\right)^2}}{\left(1-\frac{\sgn(b)\beta}{\frac{z_2}{z_1}-\frac{v_2^s}{v_1^s}}\right)\sqrt{1+\left(\frac{v^s_2}{v^s_1}\right)^2}}\right\}\right\}.
\end{equation}
Moreover,
\begin{align}\label{length_in_trapezoid_after}
L(F_{l,\delta}(Q'_3\cup Q'_1)) &= e^{-\Lambda}L(Q'_3)+\frac{\|\pmb v^s\|}{\|\pmb u(1-\beta)\|}L(Q'_1)\leq \max\left\{e^{-\Lambda}, \frac{\|\pmb v^s\|}{\|\pmb u(1-\beta)\|}\right\}L(Q'_3\cup Q'_1)\nonumber\\&\leq \max\left\{e^{-\Lambda}, \frac{\|\pmb v^s\|}{\|\pmb u(1-\beta)\|}\right\}(L(D_2D_3)+L(D_3D_4)).
\end{align}

Let $\nu'_7\in(e^{-\Lambda},1)$. Combining \eqref{length_in_trapezoid_before} and \eqref{length_in_trapezoid_after}, we obtain that there exists $\beta'_7=\beta'_7(A,\nu'_7)$ such that for any $\beta\in(0,\beta'_7)$, $L(F_{l,\delta}(Q'_3\cup Q'_1))\leq \nu'_7L(Q_1)$. 

Also, for any $\beta\in(0,\beta'_7)$ there exists $\delta_3 = \delta_3(\beta,A,l)\in(0,l)$ such that for any $\delta\in(0,\delta_3)$ we have $L(F_{l,\delta}(Q'_2))\leq e^{-\Lambda}L(Q_2)$ (see Case 3). Then, for the above choice of parameters, we have 
\begin{equation*}
L([p_1,p_2])\leq \nu'_7L(Q_1)+e^{-\Lambda}L(Q_2)\leq \nu'_7L(Q_1\cup Q_2)\leq\nu'_7d_s(\mathcal R^n).
\end{equation*}

\underline{Subcase 7.3:} Assume $q_1\in S_1$ and $q_2\in S_3$. 

Let $Q_2$ be the maximal piece of $Q$ in $\left\{(x,y)\in\mathbb T^2| m+l-\delta\leq x\leq 1\right\}$ and $Q_1=Q\setminus Q_2$. 

Let $D_1D_2D_3D_4(\pmb z)$ be a trapezoid such that $D_1D_2$ and $D_3D_4$ have directions $\pmb v^s$, $L(D_1D_2)=L(Q_2)$, $D_2D_3$ has direction $\pmb u\left(\frac{\beta l+\delta(1-\beta)}{\delta}\right)$ and $L(D_2D_3)=\delta\sqrt{1+\left(\frac{v_2^s}{v_1^s}-\frac{\sgn(b)\beta(l-\delta)}{\delta}\right)^2}$, and $D_1D_4(\pmb z)$ has direction $\pmb z = \begin{pmatrix}z_1\\z_2\end{pmatrix}$. In our case, we consider $\pmb z=\pmb v^+_{max}, \pmb v^+_{min}(-\beta)$. Moreover, to have a trapezoid for our choice of $\pmb z$, we should have $\frac{L(Q_2)}{\sqrt{1+\left(\frac{v^s_2}{v^s_1}\right)^2}}\geq\frac{\sgn(b)\beta(l-\delta)}{\frac{z_2}{z_1}-\frac{v_2^s}{v_1^s}}+\delta$. Moreover, we have 
\begin{equation*}
\sgn(b)\left(\frac{2d}{2b}-\frac{v_2^s}{v_1^s}\right) = \frac{(a+d)+\sqrt{(a+d)^2-4}}{2|b|}>0 \,\text{ and }
\end{equation*}
\begin{equation*}
\sgn(b)\left(\frac{2d-(a+d-|b|\beta)+\sqrt{(a+d-|b|\beta)^2-4}}{2b}-\frac{v_2^s}{v_1^s}\right) = \frac{|b|\beta+\sqrt{(a+d-|b|\beta)^2-4}+\sqrt{(a+d)^2-4}}{2|b|}>0.
\end{equation*}
Thus, we obtain 
\begin{equation*}
L(D_3D_4)= L(Q_2) - \delta\sqrt{1+\left(\frac{v_2^s}{v_1^s}\right)^2}\left(1+\frac{\sgn(b)\beta(l-\delta)}{\delta\left(\frac{z_2}{z_1}-\frac{v_2^s}{v_1^s}\right)}\right)\leq L(Q_2).
\end{equation*}
Then, for the choice of parameters as in Case 7.1, we have
\begin{align}\label{length_in_trapezoid_after_f}
L(F_{l,\delta}(Q'_2))+L(F_{l,\delta}(Q'_3)) &\leq \frac{\|\pmb v^s\|}{\left\|\pmb u\left(\frac{\beta l +\delta(1-\beta)}{\delta}\right)\right\|}L(Q'_2)+e^{-\Lambda}L(Q_2) \leq\nu_7L(Q_2).
\end{align}

Let $\nu''_7=\max\{\nu_7, \nu_2\}$ and $\beta''_7 = \min\{\beta_2,\beta_7\}$. Then, for any $\beta\in(0,\beta''_7)$ there exists $\delta_7=\delta_7(\beta,K,l,A)\in(0,l)$ such that for any $\delta\in(0,\delta_7)$ we have $L([p_1,p_2])\leq \nu''_7d_s(\mathcal R^n)$.

Combining Cases 1-7, we obtain the statement of the lemma.
\end{proof}

Similarly to Lemma~\ref{stable_size}, we can obtain the following lemma.

\begin{lemma}\label{unstable_size}
Consider the setting above. Then, there exists $\beta_u\in(0,\frac{a+d-2}{|b|})$ such that for any $\beta\in(0,\beta_u)$ there exist $\delta_u=\delta_u(\beta)\in(0,l)$ and $\nu_u = \nu_u (\delta_u) \in(0,1)$ with the following properties. Let $\delta\in(0,\delta_u)$ and $\mathcal R$ be the partition for $F_{l,\delta}$. Then, for any $n\in\mathbb N$, and any $R\in\mathcal R^n$ we have
$$d_u(\left(F_{l,\delta}\right)^{-1}(R))<\nu_u d_u(\mathcal R^n) .$$
\end{lemma}

\begin{proof}[Proof of Lemma~\ref{element inside}]
First, let $\beta_0=\frac{1}{2}\min\{\beta_s,\beta_u\}$, where $\beta_s$ and $\beta_u$ are as in Lemmas~\ref{stable_size} and \ref{unstable_size}. Then, $\delta' = \frac{1}{2}\min\{\delta_s(\beta_0),\delta_u(\beta_0)\}$.

Let $\delta\in(0,\delta')$. Denote $\tilde A=A\left(\frac{\beta_0l+\delta(1-\beta_0)}{\delta}\right)$. Then, for any vector $\pmb v = \alpha_1 \pmb v^-_{min}(-\beta_0)+ \alpha_2 \pmb v^-_{max}$, where $\alpha_1\alpha_2\geq 0$ and $\alpha^2_1+\alpha_2^2\neq 0$, we have
\begin{align*}
\pmb {\hat v} = \begin{pmatrix}\hat v_1\\\hat v_2 \end{pmatrix}&=\tilde A^{-1}\pmb v = \begin{pmatrix}d & -b\\ -c-\sgn(b)d\frac{\beta_0(l-\delta)}{\delta} & a+|b|\frac{\beta_0(l-\delta)}{\delta} \end{pmatrix}\left(\alpha_1\begin{pmatrix}2b\\\phi^-(a+d-|b|\beta_0) \end{pmatrix}+\alpha_2\begin{pmatrix}0\\-1 \end{pmatrix}\right)\\
&=\frac{\beta_0(l-\delta)}{\delta}|b|\left(\alpha_1\left(2d-\phi^{-}(a+d-|b|\beta_0)\right)+\alpha_2\right)\begin{pmatrix} 0\\ -1\end{pmatrix}+A^{-1}\begin{pmatrix}2b\alpha_1\\\alpha_1\phi^{-}(a+d-|b|\beta_0)-\alpha_2\end{pmatrix}.
\end{align*}

Thus, 
\begin{equation*}
\min\left\{0, \frac{2}{b\left(2d-\phi^{-}(a+d-|b|\beta_0)\right)}\right\}\leq\frac{\hat v_2}{\hat v_1}+\frac{\sgn(b)\beta_0(l-\delta)}{\delta}+\frac{a}{b}\leq \max\left\{0,\frac{2}{b\left(2d-\phi^{-}(a+d-|b|\beta_0)\right)}\right\}.
\end{equation*}

Therefore, there exists $\delta_0\in(0,\delta')$ such that the following holds. Let $\delta\in (0,\delta_0)$ and $I$ be a segment of $W^s((0,0))$ contained in $S_2(\delta)$ with endpoints $(m+l-\delta,y_1)$ and $(m+l,y_2)$ on its boundary. Then, $|y_2-y_1|\geq\frac{1}{2}\beta_0 l$.

Let $\mathcal R$ be the Markov partition for $F_{l,\delta}$. Denote by $\partial_j(\mathcal R)$ the boundary of $\mathcal R$ that is contained in $W^j((0,0))$, where $j=s,u$. By Lemmas ~\ref{Anosov diffeomorphism} and ~\ref{unstable_size}, there exists $n_1=n_1(\delta_0)\in\mathbb N$ such that there are two distinct segments $I, J\subset\partial_s(\mathcal R^n_1)\cap S_2(\delta)$ with endpoints $(m+l-\delta, y^I_1)$, $(m+l, y^I_2)$ and $(m+l-\delta, y^J_1)$, $(m+l, y^J_2)$, respectively, with the property that $\min\{|y^I_1-y^J_2|,|y^J_1-y^I_2|\}\geq\frac{1}{4}\beta_0 l$. Then, using Lemma~\ref{stable_size}, there exists also $n_2=n_2(\delta_0)\in\mathbb N$ such that there are two distinct segments $W, V\subset\partial_u(\mathcal R^n_2)\cap S_2(\delta)$ with endpoints $(m+l-\delta, y^W_1)$, $(m+l, y^W_2)$ and $(m+l-\delta, y^V_1)$, $(m+l, y^V_2)$, respectively, with the property that $W\cap I\neq\emptyset$, $W\cap J\neq\emptyset$, $V\cap I\neq\emptyset$, $V\cap J\neq\emptyset$, and the intersection does not contain the endpoints of $I,J, W, V$. Thus, for $n_0=\max\{n_1,n_2\}$ there exists $R\in\mathcal R^{n_0}$ such that $R\subset \{(x,y)\in\mathbb T^2| x\in(m+l-\delta,m+l)\}$.
\end{proof}

\subsubsection*{Markov partition for $F^w_{l,\delta}$ when $w>0$.}

We construct the Markov partition for $F^w_{l,\delta}$ in the same way as for $F_{l,\delta}$. We quickly recall it while introducing some notations. Let us draw segments of $W_w^u((0,0))$ and $W_w^s((0,0))$ until they cross sufficiently many times and separate $\mathbb T^2$ into two disjoint (curvilinear) rectangles $R_1, R_2$. Define $\mathcal R_w$ to be the partition into rectangles determined by $R_i\cap F^w_{l,\delta}(R_j)$, where $i,j=1,2$. For $n\in\mathbb N$ let $\mathcal R_w^n$ be the partition into rectangles generated by $\left(F^w_{l,\delta}\right)^i(R)\cap \left(F^w_{l,\delta}\right)^j(T)$, where $S, T\in\mathcal R_w$ and $i,j=-n,-n+1,\ldots, n-1,n$. Let $\mathcal R^0 = \mathcal R$. 

\begin{lemma}\label{element_inside_for_smooth}
Let $\beta_0, \delta_0$, and $n_0$ be as in Lemma~\ref{element inside}. Let 
\begin{equation}\label{fast strip in smoothed}
S^w_2(\delta) = \{(x,y)\in\mathbb T^2| x\in(m+l-\delta+w,m+l-w)\}. 
\end{equation}
Then, for any $\delta\in(0,\delta_0)$ there exists $w_0=w_0(\delta)\in(0,\frac{\delta}{4})$ such that there exists $R\in\mathcal R_{w_0}^{n_0}$ such that $R\subset S^{w_0}_2(\delta)$, where $\mathcal R_{w_0}$ is the Markov partition described above for $F^{w_0}_{l,\delta}$ with $\beta=\beta_0$. In particular, there exists $Q=Q(\beta_0,\delta_0)>0$ such that if $\nu^{w_0}_{l,\delta}$ is the measure of maximal entropy for $F^{w_0}_{l,\delta}$, then $\nu^{w_0}_{l,\delta}(S^{w_0}_2(\delta))\geq Q$. 
\end{lemma}

\begin{proof}
Recall that in a neighborhood of $(0,0)$ we have $F_{l,\delta}=F^w_{l,\delta}=L_A$ for $w\in(0,\frac{\delta}{4})$. In particular, the corresponding stable and unstable manifolds coincide in a neighborhood of $(0,0)$. Let $\mathcal R$ be the constructed Markov partition for $F_{l,\delta}$. From \eqref{stable_unstable_mfld_image}, there exists $N=N(n_0,\kappa)\in\mathbb N$ such that for any point $p\in\partial_j(\mathcal R^{n_0})$ for $j=s,u$ there exist $q\in W^j_{w}((0,0),\kappa)=W^j_{0}((0,0),\kappa)$ and $n\in\mathbb Z$ such that $|n|\leq N$ and $p=\left(F_{l,\delta}\right)^n(q)$. Moreover, $\left(F^w_{l,\delta}\right)^n(q)\in W^j_w((0,0))$. By Lemma~\ref{twist_function}, we obtain that $\dist\left(p,\left(F^w_{l,\delta}\right)^n(q)\right)$ can be made arbitrarily small by choosing a sufficiently small $w$, and this choice can be made in a uniform way on compact sets for $|n|\leq N$. Therefore, using Lemma~\ref{element inside}, we obtain that there exists a sufficiently small $w_0>0$ such that there exists $R\in\mathcal R_{w_0}^{n_0}$ with $R\subset S^{w_0}_2(\delta)$.

The statement about the measure of maximal entropy follows from the fact that $F^{w_0}_{l,\delta}$  is topologically semiconjugate to the topological Markov chain defined using the constructed Markov partition. Moreover, the semiconjugacy is one-to-one on all periodic points except for the fixed points. In particular, the measure of maximal entropy for $F^{w_0}_{l,\delta}$ is defined by the measure of maximal entropy for the topological Markov shift (Parry measure, see, for example, \cite[Proposition 4.4.2]{KatokHasselblatt}) using the topological semiconjugacy.
\end{proof}

\subsubsection*{Lower bound on $\lambda_{mme}(F^w_{l,\delta})$.}

We now work towards obtaining a lower bound on $\lambda_{mme}(F^w_{l,\delta})$.

We are in the setting of Lemma~\ref{element_inside_for_smooth}. We will use the same notation as in Proposition~\ref{eigenvector_eigenvalue_proposition} and Lemma~\ref{Anosov diffeomorphism}.

Let $\pmb v_\delta^u = \pmb e^+\left(\frac{\beta_0 l +\delta(1-\beta_0)}{\delta}\right)$ and $\pmb v_\delta^s = \frac{\pmb e^-\left(\frac{\beta_0 l +\delta(1-\beta_0)}{\delta}\right)}{\left\|\pmb e^-\left(\frac{\beta_0 l +\delta(1-\beta_0)}{\delta}\right)\right\|}$. Denote by $\hat\delta = a+d+|b|\left(\frac{\beta_0 l +\delta(1-\beta_0)}{\delta}-1\right)$ and $\hat\beta_0 = a+d-|b|\beta_0$.
Then,
\begin{align*}
&\pmb v^+_{max} = c_{max}^{u,\delta}\pmb v_\delta^u+c_{max}^{s,\delta}\pmb v_{\delta}^s, \text{ where }
\\ &\begin{pmatrix}c_{max}^{u,\delta}\\c_{max}^{s,\delta}\end{pmatrix} = \frac{1}{4\sqrt{\hat\delta^2-4}}\begin{pmatrix}\hat\delta+\sqrt{\hat\delta^2-4}\\\left(\sqrt{\hat\delta^2-4}-\hat\delta\right)\left\|\pmb e^-\left(\frac{\beta_0 l +\delta(1-\beta_0)}{\delta}\right)\right\|\end{pmatrix}\rightarrow \begin{pmatrix}\frac{1}{2}\\0\end{pmatrix} \text{ as } \hat\delta\rightarrow\infty, \text{ and }\\
&\pmb v^+_{min}(-\beta_0) = c_{min}^{u,\delta}\pmb v_\delta^u+c_{min}^{s,\delta}\pmb v_{\delta}^s, \text{ where } \\&\begin{pmatrix}c_{min}^{u,\delta}\\c_{min}^{s,\delta}\end{pmatrix} = \frac{1}{2\sqrt{\hat\delta^2-4}}\begin{pmatrix}\hat\delta-\hat\beta_0+\sqrt{\hat\delta^2-4}+\sqrt{\hat\beta_0^2-4}\\\left(\hat\beta_0-\hat\delta+\sqrt{\hat\delta^2-4}-\sqrt{\hat\beta_0^2-4}\right)\left\|\pmb e^-\left(\frac{\beta_0 l +\delta(1-\beta_0)}{\delta}\right)\right\| \end{pmatrix}\rightarrow\begin{pmatrix}1\\\hat\beta_0-\sqrt{\hat\beta_0^2-4}\end{pmatrix} \text{ as } \hat\delta\rightarrow\infty.
\end{align*}

As a result, if $\delta$ is sufficiently small, then for any $\pmb v = \alpha_1\pmb v^+_{max}+\alpha_2\pmb v^+_{min}(-\beta_0)$, where $\alpha_1\alpha_2\geq 0$ and $\alpha_1^2+\alpha_2^2\neq 0$, we have for any $n\in\mathbb N$

\begin{equation}\label{large expansion estimate}
\left\|\left(A\left(\frac{\beta_0l+\delta(1-\beta_0)}{\delta}\right)\right)^n\pmb v\right\|\geq \left(\mu^+\left(\frac{\beta_0l+\delta(1-\beta_0)}{\delta}\right)\right)^nC\|\pmb v\|,
\end{equation}
where $C = |b|\min\left\{\frac{1}{2\|\pmb v^+_{max}\|}, \frac{1}{\|\pmb v^+_{min}(-\beta_0)\|}\right\}\in(0,1)$. 

Let $S^{w_0}_2(\delta)$ be as in Lemma~\ref{element_inside_for_smooth}. Then, $\nu^{w_0}_{l,\delta}(S^{w_0}_2(\delta))\geq Q$. We obtain the lower bound on $\lambda_{mme}(F^w_{l,\delta})$ in a similar way as in Section~\ref{section: estimate_abs} by replacing $\mathcal S$ by $\mathbb T^2\setminus S^{w_0}_2(\delta)$ and $\mathrm Leb$ by $\nu^{w_0}_{l,\delta}$ and using Birkhoff's Ergodic Theorem for $\nu^{w_0}_{l,\delta}$. More precisely, 

Consider $p\in\mathbb T^2$ and a natural number $n$. We write
\[
n = \sum\limits_{j=1}^sn_j,
\]
where $n_1\in\{0\}\cup\mathbb N$ and the numbers $n_j\in\mathbb N$ for $j=1,\ldots,s$ are chosen in the following way:
\begin{enumerate}
    \item The number $n_1$ is the first moment when $\left(F^w_{l,\delta}\right)^{n_1}(p)\in \mathbb T^2\setminus S_2^{w_0}(\delta)$;
    \item The number $n_2$ is such that the number $n_1+n_2$ is the first moment when $\left(F^w_{l,\delta}\right)^{n_1+n_2}(p)\in S_2^{w_0}(\delta)$;
    \item The rest of the numbers are defined in the same way. For any $k\in\mathbb N$, the number $n_{2k+1}$ is such that the number $\sum\limits_{j=1}^{2k+1}n_j$ is the first moment when $\left(F^w_{l,\delta}\right)^{\sum\limits_{j=1}^{2k+1}n_j}(p)\in \mathbb T^2\setminus S_2^{w_0}(\delta)$, and the number $n_{2k+2}$ is such that the number $\sum\limits_{j=1}^{2k+2}n_j$ is the first moment when $\left(F^w_{l,\delta}\right)^{\sum\limits_{j=1}^{2k+2}n_j}(p)\in S_2^{w_0}(\delta)$.
\end{enumerate}

Let $\pmb v\in\mathcal C_p^+$ and $\|\pmb v\|=1$. Then, we have
\begin{equation*}
    \log\|D_{p}\left(F^w_{l,\delta}\right)^n\pmb v\| =\sum\limits_{j=1}^s\log\|D_{\left(F^w_{l,\delta}\right)^{n_1+n_2+\ldots+n_{j-1}}(p)}\left(F^w_{l,\delta}\right)^{n_j}\pmb v_j\|,
\end{equation*}
where $\pmb v_1=\pmb v$, $\pmb v_2 = \frac{D_{p}\left(F^w_{l,\delta}\right)^{n_1}\pmb v_1}{\|D_{p}\left(F^w_{l,\delta}\right)^{n_1}\pmb v_1\|}$, and $\pmb v_j = \frac{D_{\left(F^w_{l,\delta}\right)^{n_1+n_2+\ldots+n_{j-2}}(p)}\left(F^w_{l,\delta}\right)^{n_{j-1}}\pmb v_{j-1}}{\|D_{\left(F^w_{l,\delta}\right)^{n_1+n_2+\ldots+n_{j-2}}(p)}\left(F^w_{l,\delta}\right)^{n_{j-1}}\pmb v_{j-1}\|}$ for $j=3, \ldots, s.$ In particular, $\|\pmb v_j\|=1$ for $j=1,\ldots, s$. 

Recall that $DF^w_{l,\delta} = A\left(\frac{\beta_0l+\delta(1-\beta_0)}{\delta}\right)$ on $S_2^{w_0}(\delta)$. Thus, using \eqref{general expansion} and \eqref{large expansion estimate}, 
 we obtain for $k\in\mathbb N$
\begin{align*}
\|D_{\left(F^w_{l,\delta}\right)^{n_1+n_2+\ldots+n_{j-1}}(p)}\left(F^w_{l,\delta}\right)^{n_j}\pmb v_j\|&\geq \left(\mu^+\left(\frac{\beta_0l+\delta(1-\beta_0)}{\delta}\right)\right)^nC \qquad &\text{if}\qquad &j=2k-1,\\
\|D_{\left(F^w_{l,\delta}\right)^{n_1+n_2+\ldots+n_{j-1}}(p)}\left(F^w_{l,\delta}\right)^{n_j}\pmb v_j\|&\geq \mu^{n_j} \qquad &\text{if}\qquad &j=2k.
\end{align*}

As a result, using the fact that $\mu>1$, we have
\begin{align*}
    \log\|D_{p}\left(F^w_{l,\delta}\right)^n\pmb v\| &\geq \left[\frac{s}{2}\right]\log C+\mu^+\left(\frac{\beta_0l+\delta(1-\beta_0)}{\delta}\right)\sum\limits_{k=1}^{[\frac{s}{2}]}n_{2k-1}+(\log\mu)\sum\limits_{k=1}^{[\frac{s}{2}]}n_{2k}\\
		&\geq \left[\frac{s}{2}\right]\log C+\mu^+\left(\frac{\beta_0l+\delta(1-\beta_0)}{\delta}\right)\sum\limits_{k=1}^{[\frac{s}{2}]}n_{2k-1}.
\end{align*}

Since $F^w_{l,\delta}$ is a smooth Anosov diffeomorphism, by Birkhoff's Ergodic Theorem we obtain, that 
\begin{equation*}
\frac{1}{n}\sum\limits_{k=1}^{[\frac{s}{2}]}n_{2k-1}\rightarrow \nu^{w_0}_{l,\delta}(S_2^{w_0}(\delta)) \,\,\text{ and } \,\,\frac{1}{n}\sum\limits_{k=1}^{[\frac{s}{2}]}n_{2k}\rightarrow \left(1-\nu^{w_0}_{l,\delta}(S_2^{w_0}(\delta))\right)\,\, \text{ as }\,\, n\rightarrow\infty. 
\end{equation*}
Moreover, each visit to $\mathbb T^2\setminus S_2^{w_0}(\delta)$ is at least one iterate, so $\limsup\limits_{n\rightarrow\infty}\left(\frac{1}{n}[\frac{s}{2}]\right)\leq \left(1-\nu^{w_0}_{l,\delta}(S_2^{w_0}(\delta))\right)\leq (1-Q)$.

Therefore, we obtain
\begin{equation}\label{mme_lower_Fldelta}
\lambda_{mme}\left(F^{w_0}_{l,\delta}\right)\geq Q\log\mu^+\left(\frac{\beta_0 l+\delta(1-\beta_0)}{\delta}\right)+(1-Q)\log C \rightarrow +\infty \text{ as }\delta\rightarrow 0
\end{equation}

as $C\in(0,1)$ and $Q=Q(\beta_0,\delta_0)\in(0,1)$ are independent of $\delta$ while $w_0$ depends on $\delta$.

\section{Construction II}\label{section: decrease abs}

In this section, we show how to decrease the Lyapunov exponent with respect to the Lebesgue probability measure while controlling the Lyapunov exponent with respect to the measure of maximal entropy starting from the Anosov diffeomorphisms in Theorem~\ref{thm: increase max}. 
In this section, we use the construction described in \cite{HJJ} while providing estimates of the Lyapunov exponents. As before, we first state the main theorem of the section and give its proof before stating and proving the necessary lemmas.

\begin{theorem}\label{decrease abs for family}
Suppose $L_A$ and $\Lambda$ as in Theorem~\ref{increase max in intro}. For any $H$ such that $\Lambda<H$ and positive number $\gamma$, let $\{g_s\}_{s\in[0,1]}$ be a smooth family of area-preserving Anosov diffeomorphisms on $\mathbb T^2$ homotopic to $L_A$ from Theorem~\ref{increase max in intro} applied for $\gamma$ and $H$ with lower bound on $\lambda_{mme}(g_1)$ coming from Lemma~\ref{lemma: bound on mme in I} being larger than $H$. Then, there exists a constant $\tilde C$ such that for any $\sigma>0, S>\Lambda$ there exists 
a smooth family $\{f_{s,t}\}_{(s,t)\in[0,1]\times[0,1]}$ of Anosov diffeomorphisms on $\mathbb T^2$ homotopic to $L_A$ such that:
\begin{enumerate}[label=(\Alph*)]
\item $f_{s,0}=g_s$ for all $s\in[0,1]$;\label{A}
\item $f_{s,t}$ preserves a probability measure $\mu_{s,t}$ which is absolutely continuous with respect to the Lebesgue measure;\label{B}
\item $\lambda_{abs}(f_{s,1})<\gamma$ for all $s\in[0,1]$;\label{C}
\item $\lambda_{mme}(f_{0,t})<S$ for all $t\in[0,1]$;\label{D}
\item $\lambda_{mme}(f_{1,t})\geq H+\tilde C\sigma$.\label{E}
\end{enumerate}
\end{theorem}



\begin{proof}
We define $f_{s,0}=g_s$ for all $s\in[0,1]$ so we automatically have \ref{A} in the theorem. Moreover, by Theorem~\ref{increase max in intro}, we have that $\lambda_{abs}(f_{s,0})>\Lambda-\gamma$ and $\lambda_{mme}(f_{1,0})>H$ due to the special form of the lower bound on $\lambda_{mme}(f_{1,0})$ (see the proof of Theorem~\ref{thm: increase max}). Let $\tilde C=\log(K_1K_2^{-1})+\log(C)$ be as in Lemma~\ref{bound lambda_mme below with variables}. Choose $r_0$ small enough such that Lemma~\ref{lemma: upper bound on lambda mme start from linear} and Lemma~\ref{bound lambda_mme below with variables} hold for desired estimates. Apply Construction II to the family $\{f_{s,0}\}$ with the chosen $r_0$ and $\eta$ changing up to $\frac{\eta_1}{2}$ where $\eta_1$ comes from Lemma~\ref{lemma: decreasing abs path}. Thus, we have \ref{B} in the theorem. Furthermore, the choice of $\eta_1$ guarantees that we have \ref{C} in the theorem. Also, by the choice of $r_0$ satisfying Lemma~\ref{lemma: upper bound on lambda mme start from linear} with appropriate choice of $\chi$, we have \ref{D} in the theorem. Finally, from the form of the lower bound \ref{lower bound in lemma on mme} in Lemma~\ref{bound lambda_mme below with variables}, we obtain \ref{E} in the theorem.
\end{proof}




\subsection{Construction II: Slow-down deformation near a fixed point}\label{section slowing}

We will now describe the slow-down deformation near a fixed point that we will use. The construction comes from \cite{HJJ} but was originally introduced by A. Katok to give an example of a Bernoulli area-preserving smooth diffeomorphism (on the boundary of the set of Anosov diffeomorphism) with non-zero Lyapunov exponents on any surface (see \cite{Katokmap}). For more explanation, see Remark~\ref{remark on slow down}.

Recall that the family of diffeomorphisms, $\{f_{s,0}\}_{s\in[0,1]}$, built by Construction I has the properties that each $f_{s,0}$ has $(0,0)$ as a fixed point and is equal to $L_A$ on $\mathbb T^2\setminus \{(x,y)\in\mathbb T^2 | m-w_0<x<m+l+w_0\}$ where $m\in(0,1)$, $l\in(0,1-m)$,  and $w_0>0$ (very small). We choose a coordinate system centered at $(0,0)$ with the basis consisting of eigenvectors $\pmb v^u=\pmb e^+(1)$ and $\pmb v^s=\pmb e^-(1)$ of $A$ (see \eqref{eigenvector}). In this coordinate system, $A = \begin{pmatrix}e^\Lambda & 0\\ 0 & e^{-\Lambda}\end{pmatrix}$. 

Let $D_r = \{(s_1, s_2)| s_1^2+s_2^2\leq r^2\}$ be a disk of radius $r$ centered at $(0,0)$. Choose $0<r_0<1$ and set $r_1=2r_0\Lambda$. Then, we have
\begin{equation*}
    D_{r_0}\subset \Int A(D_{r_1})\cap \Int A^{-1}(D_{r_1}).
\end{equation*}

The linear map $x\mapsto Ax$ is the time-one map of the local flow generated by the following system of differential equations in $D_{r_1}$:

\begin{equation}\label{original flow}
    \frac{ds_1}{dt}=s_1\Lambda, \qquad \frac{ds_2}{dt} = -s_2\Lambda.
\end{equation}

Let $\alpha\in(0,\frac{1}{3})$ and $\eps\in(0,1)$ such that $\frac{1+\alpha}{(1-\eps)^{2(1+\alpha)}}<\frac{4}{3}$. Choose a $C^{\infty}$ function $\psi_0:[0,1]\rightarrow[0,1]$ satisfying:
\begin{enumerate}
    \item $\psi_0(u)=1$ for $u\geq r_0^2$;
    \item \label{property_psi0} $\psi_0(u)=\left(\frac{u}{r_0^2}\right)^{1+\alpha}$ for $0\leq u\leq\left(\frac{r_0}{2}\right)^2$\label{psi0 near 0}
    \item  $\psi_0(u)>0$ and $\psi'_0(u)\geq 0$ for $0<u<r_0^2$. \label{monotonicity psi0}
		\item $\psi_0(u)\geq \left(\frac{u}{r_0^2}\right)^{1+\alpha}$ and $\psi_0'(u)\leq \frac{1+\alpha}{(1-\eps)^{2(1+\alpha)}r_0^2}\left(\frac{u}{r_0^2}\right)^\alpha$ for $u\in[\left(\frac{r_0}{2}\right)^2, r_0^2]$; \label{stronger bounds on psi0}
    \item  In particular, for considered $\alpha$ and $\eps$, we have $\psi'_0(u)\leq \frac{4}{3r_0^2}$ when $u\in[\left(\frac{r_0}{2}\right)^2, r_0^2]$.\label{bound derivative}
\end{enumerate}
Notice that by \eqref{property_psi0} the derivative of $\psi_0$ at $\left(\frac{r_0}{2}\right)^2$ is $\frac{1+\alpha}{2^{2\alpha}r_0^2}$. See Figure~\ref{psi_graph} for an example of $\psi_0(u)$.

\begin{figure}[ht]
			\begin{subfigure}{.5\textwidth}
			\centering
\begin{tikzpicture}
\begin{axis}[
         xmax=1.3,ymax=1.3,
          axis lines=middle,
					xtick=\empty,
					ytick=\empty,
					clip=false
					]
\draw[dashed] (1/8,1)--(1/8,0) node[below] {$\left(\frac{r_0}{2}\right)^2$};
\draw[dashed] (0.5,1)--(0.5,0) node[below] {$r_0^2$};
\draw[dashed] (1,1)--(0,1) node[left] {$1$};
\draw[dashed] (1,1)--(1,0) node[below] {$1$};

\draw (0,0)--(1.25,0) node[below] {$u$};
\draw (1,0)--(0,0) node[below left] {$0$};

\addplot[red, dashed, domain=0:((1-0.03)^2)*0.5,line width=0.3mm] {pow(2*x/(1-0.03)^2,7/6)} node[above] {$\left(\frac{u}{(1-\varepsilon)^2r_0^2}\right)^{1+\alpha}$};

\addplot[domain=0:1/8,line width=0.3mm]  {pow(2*x,7/6)};
\addplot[orange, dashed,domain=1/8:0.5, line width=0.3mm]  {pow(2*x,7/6)} node[below right] {$\left(\frac{u}{r_0^2}\right)^{1+\alpha}$};
\draw[line width=0.3mm] (0.5,1)--(1,1) node[above] {$\psi_0(u)$};
\draw[smooth, tension=2, line width=0.3mm] (0.125,0.1984)--(0.27,0.5)--(0.475,0.98)--(0.48,0.987)--(0.485,0.995)--(0.5,1)--(0.6,1);
\end{axis}
\end{tikzpicture}
\caption{An example of $\psi_0(u)$.}
\label{psi_graph}
\end{subfigure}
\begin{subfigure}{.5\textwidth}
\centering
\begin{tikzpicture}
\begin{axis}[
         xmax=1.3,ymax=1.3,
          axis lines=middle,
					xtick=\empty,
					ytick=\empty,
					clip=false
					]
\draw[dashed] (0.5,1)--(0.5,0) node[below] {$r_0^2$};
\draw[dashed] (0.49,1)--(0.49,0) node[above left] {$\eta$};
\draw[dashed] (1,1)--(0,1) node[left] {$1$};
\draw[dashed] (1,1)--(1,0) node[below] {$1$};

\draw (0,0)--(1.25,0) node[below] {$u$};
\draw (1,0)--(0,0) node[below left] {$0$};

\addplot[red, dashed, domain=0:((1-0.03)^2)*0.5,line width=0.3mm] {pow(2*x/(1-0.03)^2,7/6)} node[above] {$\left(\frac{u}{(1-\varepsilon)^2r_0^2}\right)^{1+\alpha}$};

\draw[blue, line width=0.3mm] (0.5,1)--(0.15,1) node[above] {$\psi_{2r_0^2}(u)$};
\draw[blue, line width=0.3mm] (0,1)--(0.2,1);
\draw(0.01,0.46)--(0,0.46) node[left] {$\psi_0\left(\frac{\eta}{2}\right)$};

\draw[smooth,tension=3,orange,line width=0.3mm] (0,0.46)--(0.1875,0.46)--(0.21875, 0.47)--(0.225,0.475)--(0.25, 0.5)--(0.26,0.51)--(0.3, 0.58)--(0.47,0.98)--(0.475,0.987)--(0.48,0.995)--(0.5,1)--(0.6,1);
\draw[orange] (0.01,0.46)--(0.015,0.46)  node[above right] {$\psi_{\eta}(u)$};
\addplot[domain=0:1/8,line width=0.3mm]  {pow(2*x,7/6)};
\draw[line width=0.3mm] (0.5,1)--(1,1) node[above] {$\psi_0(u)$};
\draw[smooth, tension=2, line width=0.3mm] (0.125,0.1984)--(0.27,0.5)--(0.475,0.98)--(0.48,0.987)--(0.485,0.995)--(0.5,1)--(0.6,1);

\draw[dashed] (0.25,1)--(0.25,0) node[below] {$\frac{\eta}{2}$};
\draw[dashed] (0.125,1)--(0.125,0) node[below] {$\frac{\eta}{4}$};
\end{axis}
\end{tikzpicture}
\caption{An example of $\psi_{\eta}(u)$.}
\label{psi_eta_graph}
\end{subfigure}
\caption{}
\end{figure}

Define a one-parameter family of $C^{\infty}$ functions $\psi_\eta:[0,1]\rightarrow[0,1]$, where $0\leq \eta\leq 2r_0^2$, such that:
\begin{enumerate}
    \item $\psi_\eta(u)>0$ and $\psi'_\eta(u)\geq 0$ for $0<u<r_0^2$; \label{increase psi_eta}
	  \item $\psi_\eta(u)=1$ for $u\geq r_0^2$; \label{right end psi_eta}
    \item $\psi_\eta(u) = \psi_0(\frac{\eta}{2})$ for $0\leq u\leq \frac{\eta}{4}$;
    \item $\psi_\eta(u)=\psi_0(u)$ for $u\geq \eta$;
    \item if $\eta_1\leq \eta_2$, then $\psi_{\eta_1}(u)\leq \psi_{\eta_2}(u)$ for every $u\in[0,1]$;\label{monotone}
		\item \label{derivative for all eta} $\psi'_\eta(u)\leq \frac{1+\alpha}{(1-\eps)^{2(1+\alpha)}}\left(\frac{u}{r_0^2}\right)^{\alpha}\leq\frac{4}{3r_0^2}$ for $u\in(0,1)$;
    \item $\psi_\eta(u)\rightarrow \psi_0(u)$ as $\eta\rightarrow 0$ pointwise on $[0,1]$;\label{convergence}
    \item The map $(\eta,u)\mapsto \psi_\eta(u)$ is $C^\infty$ smooth.\label{differentiable}
\end{enumerate}

Notice that $\psi_{2r_0^2}(u)\equiv 1$ for $u\in[0,1]$. See Figure~\ref{psi_eta_graph} for an example of $\psi_{\eta}(u)$. Also, we have
\begin{equation}\label{integral diverge}
    \int_0^1\frac{1}{\psi_0(u)}\,du \quad \text{diverges} \quad \text{and} \quad \int_0^1\frac{1}{\psi_\eta(u)}\,du<\infty\quad\text{for}\quad \eta>0.
\end{equation}

\begin{remark}\label{remark on slow down}
Note that in \cite{Katokmap} the function $\psi_0$ (in the notation above) is such that $\int_0^1\frac{1}{\psi_0(u)}du$ converges. In comparison to \cite{HJJ} and \cite{Katokmap}, we consider a more explicit choice of $\psi_0$ which allows for the estimation of Lyapunov exponents. Furthermore, the maps that we work with are Anosov as we are not considering the map coming from $\psi_0$ itself. This is analogous to \cite[Corollary 4.2]{Katokmap}.
\end{remark}
 
Consider the following slow-down deformation of the flow described by the system \eqref{original flow} in $D_{r_0}$:
\begin{equation}\label{slow down flow}
    \frac{ds_1}{dt}=s_1\psi_\eta(s_1^2+s_2^2)\Lambda, \qquad \frac{ds_2}{dt} = -s_2\psi_\eta(s_1^2+s_2^2)\Lambda.
\end{equation}
Let $g_\eta$ be the time-one map of this flow. Observe that $g_\eta$ is defined and of class $C^\infty$ in $D_{r_1}$, and it coincides with $L_A$ in a neighborhood of $\partial D_{r_1}$ by the choice of $\psi_\eta$, $r_0$, and $r_1$. As a result, for sufficiently small $r_0$ we obtain a $C^{\infty}$ diffeomorphism 
\begin{equation}\label{map G_w}
    G_{s,\eta}(x) = \left\{
    \begin{aligned}
    &f_{s,0}(x) \quad&\text{if}\quad &x\in \mathbb T^2\setminus D_{r_1},\\
    &g_\eta(x) \quad&\text{if}\quad&x\in D_{r_1}.
    \end{aligned}
    \right.
\end{equation}


Using \eqref{integral diverge} for $\eta>0$, we can define a positive $C^{\infty}$ function
\[
\kappa_\eta(s_1,s_2):=\left\{
    \begin{aligned}
    &\left(\psi_\eta(s_1^2+s_2^2)\right)^{-1} \quad&\text{if}\quad (s_1,s_2)\in D_{r_0},\\
    &1 \quad&\text{otherwise},
    \end{aligned}
    \right.
\]

and its average

\[
K_\eta:=\int_{\mathbb T^2}\kappa_\eta\,d\Leb.
\]

For $\eta>0$ and $s\in[0,1]$, the diffeomorphism $G_{s,\eta}$ preserves the probability measure $d\mu_\eta = K_\eta^{-1}\kappa_\eta\,d\Leb$, i.e., $\mu_\eta$ is absolutely continuous with respect to the Lebesgue measure. See the paragraph containing equation (3.2) in \cite{HJJ} for the idea of the proof.

Let $p\in\mathbb T^2$. Consider the cones 
\begin{equation}\label{bigger cones}
\mathcal K^+(p) = \{\xi_1\pmb v^u+\xi_2\pmb v^s \,|\, \xi_1,\xi_2\in\mathbb R, |\xi_2|\leq |\xi_1|\} \,\text{ and } \mathcal K^-(p)= \{\xi_1\pmb v^u+\xi_2\pmb v^s \,|\, \xi_1,\xi_2\in\mathbb R, |\xi_1|\leq |\xi_2|\}
\end{equation}
in $T_p\mathbb T^2$. 

\begin{lemma}(cf. Proposition 4.1 in \cite{Katokmap})\label{invariant_cones_katokmap}
For every $s\in[0,1]$, $\eta>0$, and $p\in\mathbb T^2$ we have
\begin{equation*}
DG_{s,\eta}\mathcal K^+(p)\subsetneq \mathcal K^+(G_{s,\eta}(p)) \, \text{ and }\, DG_{s,\eta}^{-1}\mathcal K^-(p)\subsetneq \mathcal K^-(G_{s,\eta}^{-1}(p)).
\end{equation*}

Moreover, $E_{s,\eta}^+(p) = \bigcap_{j=0}^{\infty}DG^j\mathcal K^+(G^{-j}(p))$ and $E_{s,\eta}^-(p) = \bigcap_{j=0}^{\infty}DG^{-j}\mathcal K^-(G^{j}(p))$ are one-dimensional subspaces of $T_p\mathbb T^2$.
\end{lemma}

\begin{proof}
The cases for $\mathcal K^+(p)$ and $\mathcal K^-(p)$ are similar. Thus, we restrict ourselves to the inclusion for $\mathcal K^+(p)$. 

Notice that the vector $\begin{pmatrix}0\\1\end{pmatrix}$ in the standard Euclidean coordinates is equal to $\frac{1}{2\sqrt{(a+d)^2-4}}(\pmb v^u-\pmb v^s)$. As a result, using \eqref{boundaries of cone in basis of eigenvectors}, we obtain that $\mathcal C^+_p\subset \mathcal K^+(p)$ and $\mathcal C^-_p\subset \mathcal K^-(p)$ (see Lemma~\ref{Anosov diffeomorphism} for notation). By the constructions of $f_{s,0}$ and $G_{s,\eta}$ and Proposition~\ref{eigenvector_eigenvalue_proposition}, the desired inclusion holds outside of the disk $D_{r_1}$.

The system of variational equations corresponding to the system~\eqref{slow down flow} implies that the following equation holds for the tangent $\zeta_\eta$:

\begin{equation}\label{tangent equation}
\frac{d\zeta_\eta}{dt} = -2\Lambda((\psi_\eta(s_1^2+s_2^2)+(s_1^2+s_2^2)\psi'_\eta(s_1^2+s_2^2))\zeta_\eta+s_1s_2\psi_{\eta}'(s_1^2+s_2^2)(\zeta_\eta^2+1)).
\end{equation}

Since $\psi_\eta(s_1^2+s_2^2)>0$ in $D_{r_1}$ for $\eta>0$, substituting $\zeta_\eta=1$ and $\zeta_\eta=-1$ in \eqref{tangent equation} gives, respectively,
\begin{align*}
\frac{d\zeta_\eta}{dt} = -2\Lambda(\psi_\eta(s_1^2+s_2^2)+(s_1+s_2)^2\psi'_\eta(s_1^2+s_2^2))<0,\\
\frac{d\zeta_\eta}{dt} = 2\Lambda(\psi_\eta(s_1^2+s_2^2)+(s_1-s_2)^2\psi'_\eta(s_1^2+s_2^2))>0.
\end{align*}.

Thus, the result about the desired inclusion follows.

The statement that $E_{s,\eta}^+(p)$ and $E_{s,\eta}^-(p)$ are one-dimensional subspaces follows from the argument in \cite[Proposition 4.1]{Katokmap}.
\end{proof}


Lemma~\ref{invariant_cones_katokmap} and the fact that $G_{s,\eta}$ preserves a smooth positive measure imply the following.

\begin{corollary}
$G_{s,\eta}$ is a $C^\infty$ Anosov diffeomorphism on $\mathbb T^2$ for any $\eta>0$.
\end{corollary}


\subsection{Estimation of $\lambda_{abs}$ in Construction II}\label{section: abs in II}

We use the notation introduced in Section~\ref{section slowing}. The estimation of $\lambda_{abs}(G_{s,\eta})$ follows from the estimation of metric entropy in \cite[Section 3]{HJJ} which we provide here for completeness. From the following lemma, we can guarantee \eqref{decrease 2} in Theorem~\ref{thm: decrease abs}.

\begin{lemma}\label{lemma: decreasing abs path}
For any $\gamma>0$ there exists $\eta_1$ such that for any $0<\eta<\eta_1$ and $s\in[0,1]$ we have $\lambda_{abs}(G_{s,\eta})=\lambda_{\mu_{\eta}}(G_{s,\eta})<\gamma$. 
\end{lemma}
\begin{proof}
Let $U$ be any fixed neighborhood of $(0,0)$ and $U\subset D_{r_1}$. By \eqref{integral diverge} and properties \ref{monotone} and \ref{convergence} of the family of positive functions $\{\psi_\eta\}_{\eta\in[0,r_0^2]}$ defined above, applying the monotone convergence theorem, we have
\begin{equation*}
\lim\limits_{\eta\rightarrow 0}\int_{\mathbb T^2}\kappa_\eta\, d\Leb = \int_{\mathbb T^2}\lim\limits_{\eta\rightarrow 0}\kappa_\eta\, d\Leb  = \int_{\mathbb T^2}\kappa_0\, d\Leb = \infty.
\end{equation*}
and
\begin{equation*}
\lim\limits_{\eta\rightarrow 0}\int_{\mathbb T^2\setminus U}\kappa_\eta\, d\Leb = \int_{\mathbb T^2\setminus U}\lim\limits_{\eta\rightarrow 0}\kappa_\eta\, d\Leb  = \int_{\mathbb T^2\setminus U}\kappa_0\, d\Leb < \infty.
\end{equation*}

Therefore, we have the following equalities:
\begin{equation}\label{measure of complement neighborhood}
    \lim\limits_{\eta\rightarrow 0}\mu_\eta(\mathbb T^2\setminus U) =\lim\limits_{\eta\rightarrow 0}\int_{\mathbb T^2\setminus U}K_\eta^{-1}\kappa_\eta\, d\Leb = \lim\limits_{\eta\rightarrow 0}\frac{\int_{\mathbb T^2\setminus U}\kappa_\eta\,d\Leb}{\int_{\mathbb T^2}\kappa_\eta\,d\Leb} =0
\end{equation}
and
\begin{equation*}
    \lim\limits_{\eta\rightarrow 0}\mu_\eta(U) = 1.
\end{equation*}

By the ergodicity of $\mu_\eta$ for $G_{s,\eta}$ for $\eta>0$, \eqref{Lyapunov exponent through integral} implies
\begin{equation*}
    \lambda_{\mu_\eta}(G_{s,\eta}) = \int_{\mathbb T^2}\log\left|DG_{s,\eta}|_{E^u_x(G_{s,\eta})} \right|\,d\mu_\eta,
\end{equation*}
where $E^u_x(G_{s,\eta})$ is the unstable subspace at $x$ with respect to $G_{s,\eta}$.

By property \ref{psi0 near 0} of $\psi_0$ and the property \ref{differentiable} of $\psi_\eta$, we obtain that for any $\gamma>0$ there exists $\rho>0$ and $\eta_0>0$ such that $DG_{s,\eta}$ is close to the identity map and $\log\left|DG_{s,\eta}|_{E^u_x(G_{s,\eta})}\right|<\frac{\gamma}{2}$ in $D_{\rho}$ for $s\in[0,1]$ and $0<\eta<\eta_0$. Therefore, for any $s\in[0,1]$
\begin{equation}\label{estimate on small ball}
    \int_{D_\rho}\log\left|DG_{s,\eta}|_{E^u_x(G_{s,\eta})}\right|\,d\mu_\eta< \frac{\gamma}{2}.
\end{equation}

Moreover, by property \ref{differentiable}, we have that $\log\left|DG_{s,\eta}|_{E^u_x(G_{s,\eta})}\right|$ is uniformly bounded. Therefore, there exists $\eta_1<\eta_0$ such that for any $0<\eta<\eta_1$ we have

\begin{equation}\label{estimate outside}
    \int_{\mathbb T^2\setminus D_\rho}\log\left|DG_{s,\eta}|_{E^u_x(G_{s,\eta})}\right|\,d\mu_\eta< \frac{\gamma}{2}.
\end{equation}

By \eqref{estimate on small ball} and \eqref{estimate outside}, we obtain the lemma.
\end{proof}

\subsection{Estimation of $\lambda_{mme}$ in Construction II}\label{section: mme in II}

We use the notation introduced in Section~\ref{section slowing}. 


\subsubsection*{Upper bound on $\lambda_{mme}$}

Here, we prove Lemma~\ref{lemma: upper bound on lambda mme start from linear} which provides an upper bound on $\lambda_{mme}$ for a family of maps in Construction II.

The main ingredient to obtain an upper bound on $\lambda_{mme}$ is to estimate the consecutive time each trajectory spends in the annulus $D_{2\Lambda r_0}\setminus D_{\frac{r_0}{2\Lambda}}$ independent of $r_0$ and $\eta$. 

Recall that $\Lambda = h_{top}(L_A)$ and $\alpha$ is a constant that appears in the second condition on $\psi_0$ (see Section~\ref{section slowing}). 
 
\begin{lemma}(cf. Lemma 5.6 in \cite{PSZ})\label{time in annulus}
There exists $T_0>0$ depending only on $\Lambda$ and $\alpha$ such that for any solution $s_\eta(t) = (s_1(t), s_2(t))_\eta$ of \eqref{slow down flow} with $s_\eta(0)\in D_{r_1}$, we have
\begin{equation*}
    \max\{t| s_\eta(t)\in D_{r_1}\setminus D_{\frac{r_0}{2\Lambda}}\}<T_0
\end{equation*}
for any $\eta\in[0,2r_0^2]$, where $r_1=2\Lambda r_0$. 
\end{lemma}

\begin{proof}
We omit the dependence of $s_1$ and $s_2$ on $\eta$ in the notation below. Let $u=s_1^2+s_2^2$. 

Assume $s_1^2 \leq s_2^2$. Then, by \eqref{slow down flow}, we have 
\begin{align}\label{s1<s2}
    \frac{du}{dt} = 2\Lambda\psi_\eta(u)(s_1^2-s_2^2) = -2\Lambda\psi_\eta(u)(u^2-4s_1^2s_2^2)^{\frac{1}{2}}.
\end{align}

For $s_2^2\leq s_1^2$, by \eqref{slow down flow}, we have 
\begin{align}\label{s2<s1}
    \frac{du}{dt} = 2\Lambda\psi_\eta(u)(s_1^2-s_2^2) = 2\Lambda\psi_\eta(u)(u^2-4s_1^2s_2^2)^{\frac{1}{2}}.
\end{align}

Recall that by \eqref{slow down flow} we have $s_1(t)s_2(t)=s_1(0)s_2(0)$ for any $t$.

If $4s_1^2(0)s_2^2(0)\leq \frac{r_0^4}{32\Lambda^4}$, then, under the assumptions $s_1^2\leq s_2^2$ and $\left(\frac{r_0}{2\Lambda}\right)^2<u\leq r_1^2$, we have

\begin{align*}
    \frac{du}{dt} \leq -2\Lambda\psi_\eta(u)\left(\left(\frac{r_0}{2\Lambda}\right)^4-\frac{r_0^4}{32\Lambda^4}\right)^{\frac{1}{2}} = -\Lambda^{-1}\psi_\eta(u)\frac{r_0^2}{2\sqrt{2}}\leq -\Lambda^{-1}\psi_0\left(\frac{r_0^2}{(2\Lambda)^2}\right)\frac{r_0^2}{2\sqrt{2}}\leq -2^{-4-2\alpha}\Lambda^{-3-2\alpha} r_0^2.
\end{align*}

Similarly, under the assumption $s_2^2\leq s_1^2$ and $\left(\frac{r_0}{2\Lambda}\right)^2<u\leq r_1^2$, we have

\begin{align*}
    \frac{du}{dt} \geq 2\Lambda\psi_\eta(u)\left(\left(\frac{r_0}{2\Lambda}\right)^4-\frac{r_0^4}{32\Lambda^4}\right)^{\frac{1}{2}}\geq 2^{-4-2\alpha}\Lambda^{-3-2\alpha} r_0^2.
\end{align*}
Therefore, under the assumptions $s_1^2\leq s_2^2$, starting from $u(0)=r_1^2 = 4\Lambda^2r_0^2$, it takes at most $t=2^{2+2\alpha}\Lambda^{1+2\alpha}(16\Lambda^4-1)$ time to reach $u=\left(\frac{r_0}{2\Lambda}\right)^2$, unless the assumption $s_1^2\leq s_2^2$ is violated. If the assumption $s_1^2\leq s_2^2$ is violated, then, by symmetry of \eqref{s1<s2} and \eqref{s2<s1}, the orbit will leave $D_{r_1}$ in at most $2\cdot 2^{2+2\alpha}\Lambda^{1+2\alpha}(16\Lambda^4-1)$ time. A similar argument works if we start from $u(0)=\left(\frac{r_0}{2\Lambda}\right)^2$ under the assumption $s_2^2\leq s_1^2$.

Assume that $4s_1^2(0)s_2^2(0)>\frac{r_0^4}{32\Lambda^4}$. If the trajectory is in $D_{r_1}$, then $r_1^2\geq s_1^2(t)+s_2^2(t)\geq s_2(t)^2$, in particular, $\frac{r_0^4}{32\Lambda^4}<4s_1^2(t)s_2^2(t)\leq4s_1^2(t)r_1^2$, and therefore, $s_1^2(t)>\frac{r_0^2}{512\Lambda^6}$. Using \eqref{slow down flow}, we obtain
\begin{equation*}
\frac{d}{dt}\left(s_1^2\right) =  2s_1\frac{d}{dt}s_1 = 2s_1^2\psi_\eta(u)\Lambda>\frac{r_0^2}{\Lambda^{7+2\alpha}}2^{-10-2\alpha}.
\end{equation*}
Therefore, $s_1^2(t)$ will increase to $r_1^2$ and the orbit will leave $D_{r_1}$ in at most $2^{12+2\alpha}\Lambda^{9+2\alpha}$ time. 

Finally, $T_0 = \max\{2^{3+2\alpha}\Lambda^{1+2\alpha}(16\Lambda^4-1), 2^{12+2\alpha}\Lambda^{9+2\alpha}\}$.
\end{proof}

The following lemma could be of independent interest as we obtain an upper bound on the forward Lyapunov exponents for $G_{0,\eta}$. We recall that $G_{0,\eta}$ depends on the size of the ball where the slow-down deformation is done, i.e., it depends on $r_0$. Moreover, $G_{0,2r_0^2} = L_A$.

\begin{lemma}\label{lemma: upper bound on lambda mme start from linear}
For any $\chi>0$ there exists $r_\chi\in(0,1)$ such that for any $r_0\in(0,r_\chi)$, $x\in\mathbb T^2$ and $\pmb v\in T_x\mathbb T^2$ with $\|v\|\neq 0$
\begin{equation*}
\lambda(G_{0,\eta},x,\pmb v)<\Lambda+\chi
\end{equation*}
 for all $\eta\in(0,2r_0^2]$, where $\lambda(G_{0,\eta},x,\pmb v) = \limsup\limits_{n\rightarrow \infty}\frac{1}{n}\log\|D_xG_{0,\eta}^n\pmb v\|$ is the forward Lyapunov exponent of $(x,\pmb v)$ with respect to $G_{0,\eta}$.

		In particular, $\lambda_{mme}(G_{0,\eta})<\Lambda+\chi$ for all $\eta\in(0,2r_0^2]$.
\end{lemma}

\begin{proof}
Consider $x\in\mathbb T^2$ and $\pmb v\in T_x\mathbb T^2$ with $\|v\|\neq 0$. Let $n$ be a natural number. We write

$$n = \sum\limits_{j=1}^{s}n_j,$$
where the numbers $n_j\in\{0\}\cup\mathbb N$ are chosen in the following way:

\begin{enumerate}
    \item The number $n_1$ is the first moment when $G_{0,\eta}^{n_1}(x)\in D_{r_1}\setminus D_{\frac{r_0}{2\Lambda}}$;
    \item The number $n_2$ is such that the number $n_1+n_2$ is the first moment when $G_{0,\eta}^{n_1+n_2}(x)\in D_{\frac{r_0}{2\Lambda}}$;
    \item The number $n_3$ is such that the number $n_1+n_2+n_3$ is the first moment when $G_{0,\eta}^{n_1+n_2+n_3}(x)\in D_{r_1}\setminus D_{\frac{r_0}{2\Lambda}}$;
    \item The number $n_4$ is such that the number $n_1+n_2+n_3+n_4$ is the first moment when $G_{0,\eta}^{n_1+n_2+n_3+n_4}(x)\not\in D_{r_1}$;
    \item The rest of the numbers are defined following the same pattern.
\end{enumerate}

If $x\in\mathbb T^2$ is such that the $G_{0,\eta}$-orbit of $x$ does not enter into $D_{r_1}$, then $\lambda(G_{0,\eta},x,\pmb v)\leq \Lambda$ because $G_{0,\eta} = L_A$ outside of $D_{r_1}$.

Assume the $G_{0,\eta}$-orbit of $x$ enters into $D_{r_1}$. By the definition of $\lambda(G_{0,\eta},x,\pmb v)$, in that case it is enough to consider the case when $G_{0,\eta}^{-1}(x)\in D_{r_1}$ but $x\not\in D_{r_1}$. 

We have
\begin{align}\label{expansion along path}
\log\|D_xG_{0,\eta}^n\pmb v\| = \log\|\pmb v\|+ \sum\limits_{j=1}^s\log\|D_{G_{0,\eta}^{n_1+n_2+\ldots+n_{j-1}}(x)}G_{0,\eta}^{n_j}\pmb v_j\|,    
\end{align}
where $\pmb v_1=\frac{\pmb v}{\|\pmb v\|}$, $\pmb v_2 = \frac{D_xG_{0,\eta}^{n_1}\pmb v_1}{\|D_xG_{0,\eta}^{n_1}\pmb v_1\|}$, and $\pmb v_j = \frac{D_{G_{0,\eta}^{n_1+n_2+\ldots+n_{j-2}}(x)}G_{0,\eta}^{n_{j-1}}\pmb v_{j-1}}{\|D_{G_{0,\eta}^{n_1+n_2+\ldots+n_{j-2}}(x)}G_{0,\eta}^{n_{j-1}}\pmb v_{j-1}\|}$ for $j=3, \ldots, s$. In particular, $\|\pmb v_j\|=1$ for $j=1, \ldots, s$.

Recall that $G_{0,\eta}$ coincides with $L_A$ in $\mathbb T^2\setminus D_{r_1}$ (see \eqref{map G_w}). Thus, for any $N\in \mathbb N$, there exists a positive number $\theta=\theta(N, \Lambda)$ such that if $r_1<\theta$, i.e., $r_0<\frac{\theta}{2\Lambda}$, then for any $y$ such that $G_{0,\eta}^{-1}(y)\in D_{r_1}$ but $y\not\in D_{r_1}$ we have  $G_{0,\eta}^n(y)\not\in D_{r_1}$ for any $\eta\in(0,2r_0^2]$ and $0\leq n\leq N$. Therefore, if $r_0$ is sufficiently small, then $n_1\geq N$.

The coefficient matrix of the variational equation \eqref{slow down flow} is 
\begin{align}\label{matrix variational}
    C_\eta&(s_1(t), s_2(t)) =\\
    &=\Lambda\begin{pmatrix} \psi_\eta(s_1^2(t)+s_2^2(t))+2s_1^2(t)\psi'_\eta(s_1^2(t)+s_2^2(t)) & 2s_1(t)s_2(t)\psi'_\eta(s_1^2(t)+s_2^2(t))\\
    -2s_1(t)s_2(t)\psi'_\eta(s_1^2(t)+s_2^2(t)) &-\psi_\eta(s_1^2(t)+s_2^2(t))-2s_2^2(t)\psi'_\eta(s_1^2(t)+s_2^2(t))
    \end{pmatrix}.\nonumber
\end{align}

Let $s_\eta(t)=(s_1(t),s_2(t))_\eta$ be the solution of \eqref{slow down flow} with initial condition $s_\eta(0)=x$. Denote by $A_\eta(t)$ a $2\times 2$-matrix solving the variational equation
\begin{equation}\label{equation for differential}
    \frac{dA_\eta(t)}{dt}=C_\eta(s_\eta(t))A_\eta(t)
\end{equation}

with initial condition $A_\eta(0)=Id$. Then, $A_\eta(1)=D_xG_{0,\eta}$.

Moreover, by \eqref{equation for differential} and the Cauchy-Schwarz inequality, we have for any vector $\pmb v$
\begin{equation}\label{estimate of the norm}
    \frac{d\|A_\eta(t)\pmb v\|}{dt}\leq \left\|\frac{d[A_\eta(t)\pmb v]}{dt}\right\|=\left\|C_\eta(s_\eta(t))A_\eta(t)\pmb v\right\|\leq \|C_\eta(s_\eta(t))\|_{op}\|A_\eta(t)\pmb v\|,
\end{equation}
where $\|\cdot\|_{op}$ is the operator norm.

By \eqref{estimate of the norm}, \eqref{matrix variational}, the definition of $\psi_\eta$, and property \ref{bound derivative} of $\psi_0$, we have that there exists a positive constant $M$ independent of $r_0$ and $\eta$ such that $\|A_\eta(t)\pmb v\|\leq e^{Mt}\|\pmb v\|$ for any $t$ and any vector $\pmb v$. In particular, $\|D_xG_{0,\eta}\|_{op}\leq e^M$ if $x\in D_{r_1}\setminus D_{\frac{r_0}{2\Lambda}}$.

Consider $x\in D_{\frac{r_0}{2\Lambda}}$. Note that, by \eqref{slow down flow}, we have that for any $t\in[0,1]$ and $\eta\in(0,2r_0^2]$ the image of $x$ under the time-$t$ map of the flow \eqref{slow down flow} is in $D_{\frac{r_0}{2}}$. In particular, $G_{0,\eta}(D_{\frac{r_0}{2\Lambda}})\subseteq D_{\frac{r_0}{2}}$ for any $\eta\in(0,2r_0^2]$. 

Recall that if $u\in[0,\left(\frac{r_0}{2}\right)^2]$, then $\psi_0(u) = \left(\frac{u}{r_0^2}\right)^{1+\alpha}$ and $\psi'_0(u)=\frac{1+\alpha}{r_0^2}\left(\frac{u}{r_0^2}\right)^{\alpha}$. Therefore, by the choice of $\psi_\eta(u)$ for $u\in[\frac{\eta^2}{r_0^2},\eta]$, we can guarantee that $0\leq \psi_\eta(u)\leq 2^{-2-2\alpha}$ and $0\leq \psi'_\eta(u)\leq \frac{1+\alpha}{r_0^2}2^{-2\alpha}$ for $u\in[0,\left(\frac{r_0}{2}\right)^2]$ and $\eta\in(0,2r_0^2]$. Then, $\|C_\eta(s_\eta(t))\|_{op}\leq \Lambda$ if $s_\eta(t)\in D_{\frac{r_0}{2}}$ because $2^{-2-2\alpha}(3+2\alpha)\in\left(0,\frac{3}{4}\right]$ and $2^{-1-2\alpha}(1+\alpha)\in\left(0,\frac{1}{2}\right]$ for $\alpha>0$. Therefore, if $x\in D_{\frac{r_0}{2\Lambda}}$, then $\|D_xG_{0,\eta}\|_{op}\leq e^{\Lambda}$. Thus, using \eqref{expansion along path}, and Lemma~\ref{time in annulus}, we obtain that for any $\chi>0$ there exists sufficiently small $r_\chi$ such that for any $r_0\in(0,r_\chi)$, $\eta\in(0,2r_0^2]$, $x\in\mathbb T^2$, and $\pmb v\in T_x\mathbb T^2$ with $\|\pmb v\|\neq 0$, 
\begin{equation*}
    \lambda(G_{0,\eta},x,\pmb v)\leq \Lambda +\frac{2T_0M}{N}\leq \Lambda+\chi.
\end{equation*}

\end{proof}

\subsubsection*{Lower bound on $\lambda_{mme}$}

Our next and final goal is to prove Lemma~\ref{bound lambda_mme below with variables} which gives a lower bound on $\lambda_{mme}$ for the maps in Construction II.

\begin{lemma}\label{better cone}
For any $\alpha\in(0,\frac{1}{3})$ and $\eps\in(0,1)$ such that $\frac{1+\alpha}{(1-\eps)^{2(1+\alpha)}}<\frac{4}{3}$ there exists $\rho=\rho(\alpha,\eps)\in(0,1)$ such that for every $s\in[0,1]$, $\eta>0$, and $p\in D_{r_1}$ we have 
$$DG_{s,\eta}\mathcal K^+_\rho(p)\subset \mathcal K^+_{\rho}(G_{s,\eta}(p)),$$
where $\mathcal K^+_\rho(p)$ is the cone of size $\rho$ in $T_p\mathbb T^2$, i.e.,
\begin{equation}\label{cone_rho}
\mathcal K^+_\rho(p) = \{\xi_1\pmb v^u+\xi_2\pmb v^s \,|\, \xi_1,\xi_2\in\mathbb R, |\xi_2|\leq \rho|\xi_1|\}.
\end{equation}
Moreover, $\rho$ can be expressed in the following way:
\begin{equation}
\rho(\alpha,\eps) = \frac{-((1-\eps)^{2(1+\alpha)}+1+\alpha)+ \sqrt{((1-\eps)^{2(1+\alpha)}+1+\alpha)^2-(1+\alpha)^2}}{(1+\alpha)}.
\end{equation}
\end{lemma}

\begin{proof}
As in Lemma~\ref{invariant_cones_katokmap}, using the system of variational equations corresponding to the system \eqref{slow down flow}, we look at the following equation for the tangent $\zeta_\eta$:
 \begin{equation}\label{tangent equation_2}
\frac{d\zeta_\eta}{dt} = -2\Lambda((\psi_\eta(s_1^2+s_2^2)+(s_1^2+s_2^2)\psi'_\eta(s_1^2+s_2^2))\zeta_\eta+s_1s_2\psi_{\eta}'(s_1^2+s_2^2)(\zeta_\eta^2+1)).
\end{equation}
First, observe that if $(s_1,s_2)\in \left(D_{r_1}\setminus D_{r_0}\right)$ then $\psi_\eta$ is constant, in particular, $\zeta_\eta$ is decreasing when $\zeta_\eta>0$ and increasing when $\zeta_\eta<0$. Also, if $s_1s_2=0$, then we have the same conclusion about $\zeta_{\eta}$. Thus, we can assume in the consideration of the further cases that $s_1s_2\neq 0$. Due to symmetry, it is enough to analyze the case $s_1,s_2>0$.

Let $s_1,s_2>0$. Then, $\zeta_\eta$ is decreasing when $\zeta_\eta>0$, so we consider the case $\zeta_\eta<0$. Moreover, let $k = \frac{s_1s_2}{s_1^2+s_2^2}$. Notice that $k\in(0,\frac{1}{2}]$. 
 
Assume $(s_1,s_2)\in D_{r_0}$. By properties \ref{monotone} and \ref{derivative for all eta} of $\psi_\eta$, we have 
\begin{equation*}
\psi_{\eta}(s_1^2+s_2^2)\geq \left(\frac{s_1^2+s_2^2}{r_0^2}\right)^{1+\alpha} \, \text{ and } \, \psi'_{\eta}(s_1^2+s_2^2)\leq \frac{1+\alpha}{(1-\eps)^{2(1+\alpha)}r_0^2}\left(\frac{s_1^2+s_2^2}{r_0^2}\right)^\alpha. 
\end{equation*}
Thus, plugging this expression into \eqref{tangent equation_2}, we obtain
\begin{equation*}
\frac{d\zeta_\eta}{dt}\geq -2\Lambda\psi'_\eta(s_1^2+s_2^2)(s_1^2+s_2^2)\left(\left(\frac{(1-\eps)^{2(1+\alpha)}}{1+\alpha}+1\right)\zeta_\eta+k(\zeta_\eta^2+1)\right).
\end{equation*}

It is easy to see that $\frac{d\zeta_\eta}{dt}\geq 0$ if $\zeta_\eta\in[\zeta^-(k),\zeta^+(k)]$, where 
\begin{equation*}
\zeta^\pm(k) = \frac{-((1-\eps)^{2(1+\alpha)}+1+\alpha)\pm \sqrt{((1-\eps)^{2(1+\alpha)}+1+\alpha)^2-4k^2(1+\alpha)^2}}{2k(1+\alpha)}. 
\end{equation*}
Also, $\zeta^+(k)\geq \zeta^+\left(\frac{1}{2}\right)$ and $\zeta^-(k)\leq \zeta^-\left(\frac{1}{2}\right)$ for $k\in(0,\frac{1}{2}]$. Thus, $\zeta_\eta$ is non-decreasing for $\zeta_\eta\in\left[\zeta^-\left(\frac{1}{2}\right),\zeta^+\left(\frac{1}{2}\right)\right]$.

As a result, using that $\zeta_\eta$ is smooth and the above analysis, we obtain that $\rho(\alpha,\eps) = \zeta^+(\frac{1}{2})$ gives the desired cone.
\end{proof}

Let $p\in\mathbb T^2$ and $\pmb v\in T_p\mathbb T^2$. Then, $\pmb v = \xi_1 \pmb v^u+\xi_2 \pmb v^s$. Denote by $\|\cdot\|_{u,s}$ the norm in $\mathbb R^2$ such that $\|\pmb v\|_{u,s}^2 = \xi_1^2+\xi_2^2$.

\begin{lemma}\label{expanding in better cone}
Assume we are in the setting of Lemma~\ref{better cone}. For any $p\in D_{r_1}$, and any $\pmb v\in \mathcal K_{\rho(\alpha,\eps)}^+(p)$ (see \eqref{cone_rho}), we have
\begin{equation*}
\|DG_{s,\eta}\pmb v\|_{u,s}\geq\|\pmb v\|_{u,s}.
\end{equation*}

In particular, for any $p\in D_{r_1}$, and any $\pmb v\in \mathcal K_{\rho(\alpha,\eps)}^+(p)$, 
\begin{equation*}
\|DG_{s,\eta}\pmb v\|\geq K_1K_2^{-1}\|\pmb v\|,
\end{equation*} 
where $K_1,K_2$ are constants coming from the equivalence of norms $\|\cdot\|$ and $\|\cdot\|_{u,s}$ in $\mathbb R^2$, i.e., $0<K_1\leq 1\leq K_2$, $K_1\|\pmb u\|_{u,s}\leq\|\pmb u\|\leq K_2\|\pmb u\|_{u,s}$ for any $\pmb u\in\mathbb R^2$.

Moreover, let $\mathcal C_p^+$ be the union of the positive cone in the tangent space at $p\in\mathbb T^2$ spanned by $\pmb v^+_{min}(-\beta)$ and $\pmb v^+_{max}$ (see \eqref{bounds_of_cones}) and its symmetric complement. There exists $N\in\mathbb N$ such that for each $p\in\mathbb T^2$ we have 
\begin{equation*}
\left(DL_A\right)^N(\mathcal C^+_p)\subset \mathcal K^+_{\rho(\alpha,\eps)}(L_A^N(p)) \text{  and  } \left(DL_A\right)^N\left(\mathcal K^+_{\rho(\alpha,\eps)}(p)\right)\subset \mathcal C^+_{L_A^N(p)}. 
\end{equation*} 
\end{lemma}

\begin{proof}
Let $p\in D_{r_1}$ and $\pmb v(0)\in \mathcal K_\rho^+(p)$. Moreover, $\pmb v(t)$ is the evolution of $\pmb v(0)$ along the flow. In particular, $\pmb v(t) = \xi_1(t)\pmb v^u+\xi_2(t)\pmb v^s$ where $|\xi_2(t)|\leq\rho|\xi_1(t)|$ for any $t>0$.

By properties of $\psi_\eta$, we have $\frac{\psi'_\eta(s_1^2+s_2^2)}{\psi_\eta(s_1^2+s_2^2)}\leq \frac{1+\alpha}{(1-\eps)^{2(1+\alpha)}}\frac{1}{s_1^2+s_2^2}$. Using \eqref{slow down flow}, for any $\alpha\in(0,\frac{1}{3})$ and $\eps\in(0,1)$ such that $\frac{1+\alpha}{(1-\eps)^{2(1+\alpha)}}<\frac{4}{3}$ we have
\begin{align*}
\frac{d}{dt}(\xi_1^2+\xi_2^2) &= 2\Lambda\left((\psi_\eta(s_1^2+s_2^2)+2s_1^2\psi'_\eta(s_1^2+s_2^2))\xi_1^2-(\psi_\eta(s_1^2+s_2^2)+2s_2^2\psi'_\eta(s_1^2+s_2^2))\xi_2^2\right)\\
&= 2\Lambda (\psi_\eta(s_1^2+s_2^2)+2s_1^2\psi'_\eta(s_1^2+s_2^2))\left(\xi_1^2-\frac{\psi_\eta(s_1^2+s_2^2)+2s_2^2\psi'_\eta(s_1^2+s_2^2)}{\psi_\eta(s_1^2+s_2^2)+2s_1^2\psi'_\eta(s_1^2+s_2^2)}\xi_2^2\right)\\
&\geq 2\Lambda (\psi_\eta(s_1^2+s_2^2)+2s_1^2\psi'_\eta(s_1^2+s_2^2))\xi_2^2\left(\rho^{-2}(\alpha,\eps)-\left(1+2\frac{1+\alpha}{(1-\eps)^{2(1+\alpha)}}\right)\right)\geq0.
\end{align*}

Let $z\in(0,1)$. Then, for any $n\in\mathbb N$, we have $\left(DL_A\right)^n(\mathcal K^+_{z}(p))=\mathcal K^+_{e^{-2n\Lambda}z}(L_A^n(p))$. Thus, we obtain the statement about the cone inclusions.
\end{proof}

Recall that $\bar\beta$, $\bar\delta$, and $\bar w$ are the values of the parameters in the construction of $f_{1,0}$ (see \eqref{definition_twist_maps}). Let $D_r$ be a disk of radius $r$ centered at $(0,0)$ and $\bar {\mathcal S}= \{(x,y)\in\mathbb T^2| x\in(m+l-\bar\delta+\bar w,m+l-\bar w)\}$, i.e., the region of the perturbation described in Construction I where $DG_{1,\eta}=A\left(\frac{\bar\beta+\bar\delta(1-\bar\delta)}{\bar \delta}\right)$ (see Proposition~\ref{eigenvector_eigenvalue_proposition}). Denote by $\nu_{G_{1,\eta}}$ the measure of maximal entropy for $G_{1,\eta}$.

\begin{lemma}\label{bound lambda_mme below with variables}
Let $\alpha\in(0,\frac{1}{3})$ and $\eps\in(0,1)$ such that $\frac{1+\alpha}{(1-\eps)^{2(1+\alpha)}}<\frac{4}{3}$. Let $\bar Q=\nu_{f_{1,0}}(\bar{\mathcal S})$. Then, for any $\sigma>0$ there exists $r_{\sigma}$ such that for all sufficiently small $r_0\in(0,r_{\sigma})$ and for all $\eta\in(0,2r_0^2]$ the following hold for $G_{1,\eta}$ obtained in Construction II with the parameters $\alpha, \eps, r_0$, and $\eta$:
\begin{enumerate}
\item $\nu_{G_{1,\eta}}(D_{r_\sigma})\leq \sigma$;
\item\label{lower bound in lemma on mme} $\lambda_{mme}(G_{1,\eta})\geq \left(\log\left(K_1K_2^{-1}\right)+\log(C)\right)\sigma+\log(C)(1-\bar Q)+\log\mu^+\left(\frac{\bar\beta+\bar\delta(1-\bar\delta)}{\bar\delta}\right)(\bar Q-\sigma),$

where $C$ is a constant that depends only on the matrix $A$ and $\bar\beta$  (see \eqref{large expansion estimate}), and $K_1, K_2$ are constants in Lemma~\ref{expanding in better cone}.
\end{enumerate}
\end{lemma}

\begin{proof}
Recall that in Construction II, for any sufficiently small $r_0\in(0,1)$ we have that $f_{1,0}=G_{1,2r_0^2}$ and for any $\eta\in(0,2r_0^2]$, $G_{1,\eta}(x)=f_{1,0}(x)$ if $x\in\mathbb T^2\setminus D_{r_1}$, where $r_1=2r_0\Lambda$. Also, there exists $\bar r>0$ such that for any $r\in(0,\bar r)$, $f_{1,0}(\bar {\mathcal S})\cap D_{r}=\emptyset$ and $f_{1,0}(D_{r})\cap \bar{\mathcal S}=\emptyset$. 

Consider a periodic point $q$ of $f_{1,0}$ other than $(0,0)$. Build a Markov partition $\mathcal {MP}$ for $f_{1,0}$ using the point $q$ and its stable and unstable manifolds, $W^s(q)$ and $W^u(q)$, respectively (Adler-Weiss construction). Since $(0,0)$ is a fixed point for $f_{1,0}$, then $(0,0)\not\in W^s(q)\cap W^u(q)$. In particular, there is a refinement $\overline{ \mathcal{MP}}$ of $\mathcal {MP}$ such that 
\begin{itemize}
\item if $\mathcal P\bar{\mathcal S} = \{R\in \overline{ \mathcal{MP}}| R\subset \bar{\mathcal{S}}\}$, then $\nu_{f_{1,0}}(\bigcup\limits_{R\in\mathcal P\bar{\mathcal S}}R)\geq \bar Q-\sigma$;
\item there exists $R_{\sigma}\in\overline{\mathcal{MP}}$ such that $(0,0)$ is in the interior of $R_\sigma$ and $\nu_{f_{1,0}}(R_\sigma)<\sigma$.
\end{itemize}

Choose $r_\sigma<\bar r$ such that $D_{r_\sigma}\subset R_\sigma$. 

Let $\rho(\alpha,\eps)$ be as in Lemma~\ref{better cone}. Let $\mathcal C_p^+$ be the union of the positive cone in the tangent space at $p\in\mathbb T^2$ spanned by $\pmb v^+_{min}(-\bar\beta)$ and $\pmb v^+_{max}$ (see \eqref{bounds_of_cones}) and its symmetric complement, where $\bar\beta$ is the value of $\beta$ in the construction of $f_{1,0}$. By Lemma~\ref{expanding in better cone}, we can choose $N\in\mathbb N$ such that for each $p\in\mathbb T^2$ we have 
\begin{equation*}
\left(DL_A\right)^N(\mathcal C^+_p)\subset \mathcal K^+_{\rho(\alpha,\eps)}(L_A^N(p)) \text{  and  } \left(DL_A\right)^N\left(\mathcal K^+_{\rho(\alpha,\eps)}(p)\right)\subset \mathcal C^+_{L_A^N(p)}. 
\end{equation*} 

Let $r_0>0$ such that $D_{r_1}\subset D_{r_\sigma}$. Recall that $r_1=2r_0\Lambda$. Since $f_{1,0}=L_A$ in a neighborhood of $(0,0)$, we can choose sufficiently small $r_0$ such that the following two facts hold:
\begin{itemize}
\item Let $p\in\mathbb T^2$ and $k, n\in\{0\}\cup\mathbb N$ be such that $f_{1,0}^{k}(p)\not\in D_{r_\sigma}$, $f_{1,0}^{k+j}(p)\in D_{r_\sigma}\setminus D_{r_1}$ for $j=1,2,\ldots, n$, and $f_{1,0}^{k+n+1}(p)\in D_{r_1}$. Then, $n\geq N$.

\item  Let $p\in\mathbb T^2$ and $k, n\in\{0\}\cup\mathbb N$ be such that $f_{1,0}^{k}(p)\in D_{r_1}$, $f_{1,0}^{k+j}(p)\in D_{r_\sigma}\setminus D_{r_1}$ for $j=1,2,\ldots, n$, and $f_{1,0}^{k+n+1}(p)\not\in D_{r_\sigma}$. Then, $n\geq N$.
\end{itemize}

Let $G_{1,\eta}$ be an Anosov diffeomorphism obtained from $f_{1,0}$ using Construction II with the parameter $r_0$ as above. Then, for any $\eta\in(0,2r_0^2]$ we have the following:
\begin{itemize}
\item $\nu_{G_{1,\eta}}(D_{r_\sigma})<\sigma$;
\item $\nu_{G_{1,\eta}}(\bar{\mathcal S})\geq\bar Q-\sigma$;
\item Let $p\in\mathbb T^2$ and $k, n\in\{0\}\cup\mathbb N$ be such that $G_{1,\eta}^{k}(p)\not\in D_{r_\sigma}$, $G_{1,\eta}^{k+j}(p)\in D_{r_\sigma}\setminus D_{r_1}$ for $j=1,2,\ldots, n$, and $G_{1,\eta}^{k+n+1}(p)\in D_{r_1}$. Then, $n\geq N$.

\item  Let $p\in\mathbb T^2$ and $k, n\in\{0\}\cup\mathbb N$ be such that $G_{1,\eta}^{k}(p)\in D_{r_1}$, $G_{1,\eta}^{k+j}(p)\in D_{r_\sigma}\setminus D_{r_1}$ for $j=1,2,\ldots, n$, and $G_{1,\eta}^{k+n+1}(p)\not\in D_{r_\sigma}$. Then, $n\geq N$.
\end{itemize}

Consider $p\in\mathbb T^2\setminus D_{r_\sigma}$ and a natural number $n$. We write $n=\sum\limits_{j=1}^s n_j$, where the numbers $n_j\in\{0\}\cup\mathbb N$ are chosen in the following way (see Figure~\ref{partition depending on regions} for an example):
\begin{enumerate}
\item The number $n_1$ is the first moment when $\left(G_{1,\eta}\right)^{n_1}(p)\in \bar{\mathcal S}\cup D_{r_\sigma}$;
\item The number $n_2$ is such that the number $n_1+n_2$ is the first moment when 

$\left(G_{1,\eta}\right)^{n_1+n_2}(p)\in \mathbb T^2\setminus\left(\bar{\mathcal S}\cup D_{r_\sigma}\right)$;

\item The rest of the numbers are defined following this pattern as done in the proof of Lemma \ref{lemma: upper bound on lambda mme start from linear}.
\end{enumerate} 

Notice that by the choice of $r_\sigma$, for any $j=1, 3, 5, \ldots$ we have that 

\begin{center}
either $\left(G_{1,\eta}\right)^{n_1+n_2+\ldots+n_j+k}(p)\in \bar{\mathcal S}$ for all $k\in\mathbb Z\cap[0,n_{j+1})$ 

or $\left(G_{1,\eta}\right)^{n_1+n_2+\ldots+n_j+k}(p)\in D_{r_\sigma}$ for all $k\in\mathbb Z\cap[0,n_{j+1})$.
\end{center}

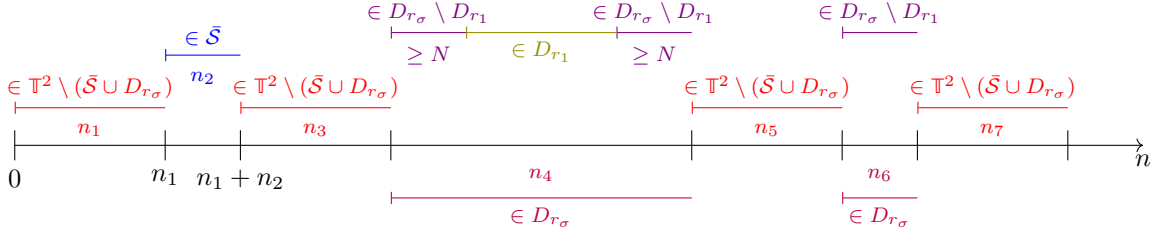
\begin{figure}
\centering
\begin{tikzpicture}
\draw[->] (0,0)--(15,0) node[below] {$n$};
\draw (0,0.2)--(0,-0.2) node[below] {$0$};
\draw (2,0.2)--(2,-0.2) node[below] {$n_1$};

\draw[red, |-] (0,0.5)--(2,0.5) node[above, midway] {\footnotesize $\in \mathbb{T}^2\setminus (\bar{\mathcal S}\cup D_{r_\sigma})$};
\node[red] at (1,0.2) {\footnotesize $n_1$};

\draw (3,0.2)--(3,-0.2) node[below] {$n_1+n_2$};
\draw[blue, |-] (2,1.2)--(3,1.2) node[above, midway] {\footnotesize $\in \bar{\mathcal S}$};
\node[blue] at (2.5,0.9) {\footnotesize $n_2$};

\draw (5,0.2)--(5,-0.2);

\draw[red, |-] (3,0.5)--(5,0.5) node[above, midway] {\footnotesize $\in \mathbb{T}^2\setminus (\bar{\mathcal S}\cup D_{r_\sigma})$};
\node[red] at (4,0.2) {\footnotesize $n_3$};

\draw (9,0.2)--(9,-0.2);

\draw[violet, |-] (5,1.5)--(6,1.5) node[above, midway] {\footnotesize $\in D_{r_\sigma}\setminus D_{r_1}$};
\node[violet] at (5.5,1.2) {\footnotesize $\geq N$};

\draw[violet, |-] (8,1.5)--(9,1.5) node[above, midway] {\footnotesize $\in D_{r_\sigma}\setminus D_{r_1}$};
\node[violet] at (8.5,1.2) {\footnotesize $\geq N$};

\draw[olive, |-] (6,1.5)--(8,1.5) node[below, midway] {\footnotesize $\in D_{r_1}$};

\draw[purple, |-] (5,-0.7)--(9,-0.7) node[below, midway] {\footnotesize $\in D_{r_\sigma}$};
\node[purple] at (7,-0.4) {\footnotesize $n_4$};

\draw (11,0.2)--(11,-0.2);

\draw[red, |-] (9,0.5)--(11,0.5) node[above, midway] {\footnotesize $\in \mathbb{T}^2\setminus (\bar{\mathcal S}\cup D_{r_\sigma})$};
\node[red] at (10,0.2) {\footnotesize $n_5$};

\draw (12,0.2)--(12,-0.2);

\draw[purple, |-] (11,-0.7)--(12,-0.7) node[below, midway] {\footnotesize $\in D_{r_\sigma}$};
\node[purple] at (11.5,-0.4) {\footnotesize $n_6$};

\draw[violet, |-] (11,1.5)--(12,1.5) node[above, midway] {\footnotesize $\in D_{r_\sigma}\setminus D_{r_1}$};

\draw (14,0.2)--(14,-0.2);

\draw[red, |-] (12,0.5)--(14,0.5) node[above, midway] {\footnotesize $\in \mathbb{T}^2\setminus (\bar{\mathcal S}\cup D_{r_\sigma})$};
\node[red] at (13,0.2) {\footnotesize $n_7$};

\end{tikzpicture}
\caption{Partition of the orbit of $p$ under $G_{1,\eta}$.}
\label{partition depending on regions}
\end{figure}

Let $\pmb v\in\mathcal C_p^+$ and $\|\pmb v\|=1$. Then, we have
\begin{equation*}
    \log\|D_{p}\left(G_{1,\eta}\right)^n\pmb v\| =\sum\limits_{j=1}^s\log\|D_{\left(G_{1,\eta}\right)^{n_1+n_2+\ldots+n_{j-1}}(p)}\left(G_{1,\eta}\right)^{n_j}\pmb v_j\|,
\end{equation*}
where $\pmb v_1=\pmb v$, $\pmb v_2 = \frac{D_{p}\left(G_{l,\eta}\right)^{n_1}\pmb v_1}{\|D_{p}\left(G_{1,\eta}\right)^{n_1}\pmb v_1\|}$, and $\pmb v_j = \frac{D_{\left(G_{1,\eta}\right)^{n_1+n_2+\ldots+n_{j-2}}(p)}\left(G_{1,\eta}\right)^{n_{j-1}}\pmb v_{j-1}}{\|D_{\left(G_{1,\eta}\right)^{n_1+n_2+\ldots+n_{j-2}}(p)}\left(G_{1,\eta}\right)^{n_{j-1}}\pmb v_{j-1}\|}$ for $j=3, \ldots, s.$ In particular, $\|\pmb v_j\|=1$ for $j=1,\ldots, s$. 

Using \eqref{general expansion}, \eqref{large expansion estimate}, and Lemma~\ref{expanding in better cone}, we obtain for $k\in\mathbb N$
\begin{equation*}
\|D_{\left(G_{1,\eta}\right)^{n_1+n_2+\ldots+n_{j-1}}(p)}\left(G_{1,\eta}\right)^{n_j}\pmb v_j\|\geq\left\{
\begin{aligned}
&K_1K^{-1}_2 \quad &\text{if}\qquad &j=2k,\left(G_{1,\eta}\right)^{n_1+n_2+\ldots+n_{j-1}}(p)\in D_{r_\sigma},\\
&\left(\mu^+\left(\frac{\bar\beta+\bar\delta(1-\bar\delta)}{\bar\delta}\right)\right)^{n_j}C \quad &\text{if}\qquad &j=2k, \left(G_{1,\eta}\right)^{n_1+n_2+\ldots+n_{j-1}}(p)\in \bar{\mathcal S},\\
&\mu^{n_j} \quad &\text{if}\qquad &j=2k-1,
\end{aligned}\right.
\end{equation*}
where $C$ is a constant that depends only on $A$ and $\bar\beta$ (see \eqref{large expansion estimate}).

As a result,
\begin{align*}
    \log\|D_{p}\left(G_{1,\eta}\right)^n\pmb v\| &\geq \log\left(K_1K_2^{-1}\right)\sum\limits_{k=1}^{[\frac{s}{2}]}\mathbb{1}_{D_{r_\sigma}}\left(\left(G_{1,\eta}\right)^{n_1+n_2+\ldots+n_{2k-1}}(p)\right)\\&+\log\left(C\right)\sum\limits_{k=1}^{[\frac{s}{2}]}\mathbb{1}_{\bar{\mathcal S}}\left(\left(G_{1,\eta}\right)^{n_1+n_2+\ldots+n_{2k-1}}(p)\right)\\&+\log\mu^+\left(\frac{\bar\beta+\bar\delta(1-\bar\delta)}{\bar\delta}\right)\sum\limits_{k=1}^{[\frac{s}{2}]}n_{2k}\mathbb{1}_{\bar{\mathcal S}}\left(\left(G_{1,\eta}\right)^{n_1+n_2+\ldots+n_{2k-1}}(p)\right)\\&+(\log\mu)\sum\limits_{k=1}^{[\frac{s}{2}]}n_{2k-1},
\end{align*}
where $\mathbb{1}_{D_{r_\sigma}}$ and $\mathbb{1}_{\bar{\mathcal S}}$ are the characteristic functions of the corresponding sets.

Since $G_{1,\eta}$ is a smooth Anosov diffeomorphism, using Birkhoff's Ergodic Theorem, we obtain as $n\rightarrow \infty$
\begin{equation*}
\frac{1}{n}\sum\limits_{k=1}^{[\frac{s}{2}]}n_{2k}\mathbb{1}_{\bar{\mathcal S}}\left(\left(G_{1,\eta}\right)^{n_1+n_2+\ldots+n_{2k-1}}(p)\right) \rightarrow \nu_{G_{1,\eta}}(\bar{\mathcal S}) \quad\text{and}\quad \frac{1}{n}\sum\limits_{k=1}^{[\frac{s}{2}]}n_{2k-1}\rightarrow \nu_{G_{1,\eta}}(\mathbb T^2\setminus(D_{r_\sigma}\cup \bar{\mathcal S})).
\end{equation*}
Moreover, since $n_{2k}\geq 1$ for $k=1,2,\ldots, \left[\frac{s}{2}\right]$, we have
\begin{equation*}
\lim\limits_{n\rightarrow\infty}\frac{1}{n}\sum\limits_{k=1}^{[\frac{s}{2}]}\mathbb{1}_{D_{r_\sigma}}\left(\left(G_{1,\eta}\right)^{n_1+n_2+\ldots+n_{2k-1}}(p)\right)\leq\lim\limits_{n\rightarrow\infty}\sum\limits_{k=1}^{[\frac{s}{2}]}n_{2k}\mathbb{1}_{D_{r_\sigma}}\left(\left(G_{1,\eta}\right)^{n_1+n_2+\ldots+n_{2k-1}}(p)\right)=\nu_{G_{1,\eta}}(D_{r_\sigma}).
\end{equation*}
Also, notice that each visit to $\mathbb T^2\setminus(D_{r_\sigma}\cup\bar{\mathcal S})$ is at least one iterate, so
\begin{equation*}
\lim\limits_{n\rightarrow\infty}\frac{1}{n}\sum\limits_{k=1}^{[\frac{s}{2}]}\mathbb{1}_{\bar{\mathcal S}}\left(\left(G_{1,\eta}\right)^{n_1+n_2+\ldots+n_{2k-1}}(p)\right)\leq\lim\limits_{n\rightarrow\infty}\frac{1}{n}\left[\frac{s}{2}\right]\leq \nu_{G_{1,\eta}}(\mathbb T^2\setminus(D_{r_\sigma}\cup\bar{\mathcal S}))=1-\nu_{G_{1,\eta}}(\bar{\mathcal S})-\nu_{G_{1,\eta}}(D_{r_\sigma}).
\end{equation*}

Thus, since $K_1K_2^{-1}, C\in(0,1)$,
\begin{align*}
\lambda_{mme}(G_{1,\eta})&\geq \log\left(K_1K_2^{-1}\right)\nu_{G_{1,\eta}}(D_{r_\sigma})+\log(C)(1-\nu_{G_{1,\eta}}(\bar{\mathcal S})-\nu_{G_{1,\eta}}(D_{r_\sigma}))\\&+\log\mu^+\left(\frac{\bar\beta+\bar\delta(1-\bar\delta)}{\bar\delta}\right)\nu_{G_{1,\eta}(\bar{\mathcal S})}+(\log\mu)\nu_{G_{1,\eta}}(\mathbb T^2\setminus(D_{r_\sigma}\cup\bar{\mathcal S}))\\&\geq \log\left(K_1K_2^{-1}\right)\sigma+\log(C)(1-\bar Q+\sigma)+(\bar Q-\sigma)\log\mu^+\left(\frac{\bar\beta+\bar\delta(1-\bar\delta)}{\bar\delta}\right).
\end{align*} 

\end{proof}
\bibliography{bibliography}{}

\begin{thebibliography}{BKRH19}

\bibitem[AS67]{AnosovSinai}
Dmitri Anosov and Yakov Sinai.
\newblock Certain smooth ergodic systems.
\newblock {\em Russ. Math. Surv.}, 22:103--167, 1967.

\bibitem[AW67]{AdlerWeiss}
Roy Adler and Benjamin Weiss.
\newblock Entropy, a complete metric invariant for automorphisms of the torus.
\newblock {\em Proc. Nat. Acad. Sci. U.S.A.}, 57:1573--1576, 1967.

\bibitem[BG14]{BG}
Keith Burns and Katrin Gelfert.
\newblock Lyapunov spectrum for geodesic flows of rank 1 surfaces.
\newblock {\em Discrete Contin. Dyn. Syst.}, 34(5):1841--1872, 2014.

\bibitem[BKRH19]{BKRH}
Jairo Bochi, Anatole Katok, and Federico Rodriguez~Hertz.
\newblock Flexibility of {L}yapunov exponents.
\newblock {\em ArXiv: 1908.07891}, 2019.

\bibitem[Bow71]{Bowen}
Rufus Bowen.
\newblock Periodic points and measures for {A}xiom {A} diffeomorphisms.
\newblock {\em Trans. Amer. Math. Soc.}, 154:377--397, 1971.

\bibitem[BP07]{BarreiraPesin}
Luis Barreira and Yakov Pesin.
\newblock {\em Nonuniform hyperbolicity: dynamics of systems with nonzero
  Lyapunov exponents}, volume 115 of {\em Encyclopedia of Mathematics and its
  Applications}.
\newblock Cambridge University Press, Cambridge, 2007.

\bibitem[dlL87]{delaLlave}
Rafael de~la Llave.
\newblock Invariant for smooth conjugacy of hyperbolic dynamical systems {II}.
\newblock {\em Commun. Math. Phys.}, 109:369--378, 1987.

\bibitem[Erc19]{E}
Alena Erchenko.
\newblock Flexibility of {L}yapunov exponents for expanding circle maps.
\newblock {\em Discrete Contin. Dyn. Syst.}, 39(5):2325--2342, 2019.

\bibitem[HJJ17]{HJJ}
Huyi Hu, Miaohua Jiang, and Yunping Jiang.
\newblock Infimum of the metric entropy of volume preserving {A}nosov systems.
\newblock {\em Discrete Contin. Dyn. Syst.}, 37(9):4767--4783, 2017.

\bibitem[Kat79]{Katokmap}
Anatole Katok.
\newblock Bernoulli diffeomorphisms on surfaces.
\newblock {\em Ann. of Math. (2)}, 110(3):529--547, 1979.

\bibitem[KH95]{KatokHasselblatt}
Anatole Katok and Boris Hasselblatt.
\newblock {\em Introduction to the modern theory of dynamical systems},
  volume~54 of {\em Encyclopedia of Mathematics and its Applications}.
\newblock Cambridge University Press, Cambridge, 1995.
\newblock With a supplementary chapter by Katok and Leonardo Mendoza.

\bibitem[KS01]{KotaniSunada}
Motoko Kotani and Toshikazu Sunada.
\newblock The pressure and higher correlations for an {A}nosov diffeomorphism.
\newblock {\em Ergod. Th.\& Dynam. Sys.}, 21:807--821, 2001.

\bibitem[MM87]{MM}
Jos\'e~Manuel Marco and Roberto Moriy\'on.
\newblock Invariants for smooth conjugacy of hyperbolic dynamical systems,
  {III}.
\newblock {\em Commun. Math. Phys.}, 112(2):317--333, 1987.

\bibitem[Mos69]{Moser}
J\"urgen Moser.
\newblock On a theorem of {A}nosov.
\newblock {\em J. Differential Equations}, 5:411--440, 1969.

\bibitem[PSZ19]{PSZ}
Yakov Pesin, Samuel Senti, and Ke~Zhang.
\newblock Thermodynamics of the {K}atok {M}ap.
\newblock {\em Ergod. Th. \& Dynam. Sys.}, 39:764--794, 2019.

\bibitem[Rue78a]{Ruelle}
David Ruelle.
\newblock An inequality for the entropy of differentiable maps.
\newblock {\em Bull. Braz. Math. Soc.}, 9:83--87, 1978.

\bibitem[Rue78b]{RuelleThF}
David Ruelle.
\newblock {\em Thermodynamic Formalism}.
\newblock Addison-Wesley, 1978.

\end{thebibliography}
\bibliographystyle{alpha}

\Addresses
\end{document}